\documentclass[12pt,reqno]{amsart}
\usepackage{mathrsfs}
\usepackage{amsfonts}
\usepackage{latexsym,amssymb,amsmath,amsthm}
\usepackage{amsmath,amsthm,amssymb}
\usepackage{amsmath}
\usepackage{a4wide}
\usepackage[numbers,sort&compress]{natbib}
\usepackage{marvosym}
\usepackage{lineno}
\usepackage{eurosym}
\usepackage{graphicx}
\usepackage{subfigure} 
\usepackage[colorlinks]{hyperref}
\hypersetup{
	colorlinks=true,
	linkcolor=blue,
	filecolor=gray,
	urlcolor=blue,
	citecolor=blue,
}
\usepackage{fancyhdr}
\allowdisplaybreaks
\numberwithin{equation}{section}
\allowdisplaybreaks[4]

\numberwithin{equation}{section} \theoremstyle{plain}
\newtheorem{theorem}{Theorem}[section]%
\newtheorem{lemma}[theorem]{Lemma}%
\newtheorem{remark}[theorem]{Remark}%
\newcommand{\R}{{\mathbb R}}

\newcommand{\ds}{\displaystyle}

\begin{document}
\title
{The existence of positive solution for an elliptic problem with critical growth and logarithmic perturbation
}

\maketitle

\begin{center}
\author{Yinbin Deng}
\footnote{ School of Mathematics  and  Statistics, Central China Normal University, Wuhan 430079, China.  Email: \texttt{ybdeng@ccnu.edu.cn}.},
\author{Qihan He}
\footnote{ College of Mathematics and Information Science, Guangxi Center for Mathematical Research,
Guangxi University,
Nanning, 530003, China, Email: \texttt{heqihan277@gxu.edu.cn}.},
\author{Yiqing Pan}
\footnote{ College of Mathematics and Information Science, Guangxi Center for Mathematical Research,
Guangxi University,
Nanning, 530003, China, Email: \texttt{13718049940@163.com}. }
and
\author{Xuexiu Zhong}
\footnote{ South China Research Center for Applied Mathematics and Interdisciplinary Studies,
South China Normal University,
Guangzhou 510631, China, Email: \texttt{zhongxuexiu1989@163.com}. }
\end{center}

\begin{abstract} We consider the existence and nonexistence of positive solution for the following Br\'ezis-Nirenberg problem with logarithmic perturbation:
    \begin{equation*}
       \begin{cases}
          -\Delta u={\left|u\right|}^{{2}^{\ast }-2}u+\lambda u+\mu u\log {u}^{2} &x\in \Omega,\\
         \quad \;\:\, u=0& x\in \partial \Omega,
       \end{cases}
    \end{equation*}
   where $\Omega$ $\subset$ $\R^N$ is a bounded smooth domain, $\lambda, \mu \in \R$, $N\ge3$ and ${2}^{\ast }:=\frac{2N}{N-2}$ is the critical  Sobolev exponent for the embedding $H^1_{0}(\Omega)\hookrightarrow L^{2^\ast}(\Omega)$. The uncertainty of the sign of $s\log s^2$ in $(0, +\infty)$ has some interest in itself. We will show the existence of positive ground state solution which is of mountain pass type provided  $\lambda\in \R, \mu>0$ and $N\geq 4$. While the case of $\mu<0$ is thornier. However, for $N=3,4$ $\lambda\in (-\infty, \lambda_1(\Omega))$, we can also establish the existence of positive  solution under some further suitable assumptions.  And a nonexistence result is also obtained for  $\mu<0$ and $-\frac{(N-2)\mu}{2}+\frac{(N-2)\mu}{2}\log(-\frac{(N-2)\mu}{2})+\lambda-\lambda_1(\Omega)\geq 0$ if $N\geq 3$. Comparing with the results in Br\'ezis, H. and  Nirenberg, L. (Comm. Pure Appl. Math.  1983), some new interesting phenomenon occurs when the parameter $\mu$ on logarithmic perturbation is not zero.

     \end{abstract}

    \section{Introduction and main results}
     \indent In this paper, we investigate the existence and nonexistence  of positive solution for the following  Br\'ezis-Nirenberg problem with a logarithmic term:
     \begin{equation}\label{1.1}
       \begin{cases}
           -\Delta u={\left|u\right|}^{{2}^{\ast }-2}u+\lambda u+\mu u\log {u}^{2} & x\in \Omega,\\
           \quad \;\:\,  u=0&x\in \partial\Omega,
       \end{cases}
     \end{equation}
     where $\Omega$ $\subset$ $\R^N$ is a bounded smooth domain, $\lambda, \mu \in \R$, $N\ge3$, and
     ${2}^{\ast }=\frac{2N}{N-2}$ is the critical Sobolev exponent for the embedding $H^1_{0}(\Omega)\hookrightarrow L^{2^\ast}(\Omega)$. Here $H^1_{0}(\Omega)$ denotes the closure of $C^\infty_0(\Omega)$ equipped with the norm $\left\|u\right\|:=(\int_\Omega |\nabla u|^2 dx)^\frac{1}{2}$.

     Our motivation to consider \eqref{1.1} is that it resembles some variational problems in geometry and physics,  which is lack of compactness. The most notorious example is Yamabe's problem: finding a function $u$ satisfying
    $$
       \begin{cases}
           -4\frac{N-1}{N-2}\Delta u=R^\prime{\left|u\right|}^{{2}^{\ast }-2}u-R(x) u & \text{ }~\hbox{on}~\mathcal{M},\\
            \quad \;\:\,\quad \;\:\,\;\:\, u>0&\text{ }~\hbox{on}~\mathcal{M},
       \end{cases}
    $$
  where  $R^\prime$ is a constant, $\mathcal{M}$ is  an $N$-dimensional Riemannian manifold, $\Delta$ denotes the Laplacian and $R(x)$ represents the scalar curvature. Some other examples we refer to \cite{bre1,lie1,lio,tau1,tau2,uhl} and the references therein.

    \indent When $\lambda=\mu=0$,  Eq.\eqref{1.1} is reduced to
    \begin{equation}\label{eq:20220524-e1}
       \begin{cases}
          -\Delta u={\left|u\right|}^{{2}^{\ast }-2}u& \text{ }x\in{\Omega },\\
          \quad \;\:\, u=0&\text{ } x\in{\partial \Omega }.
       \end{cases}
    \end{equation}
Pohozaev \cite{poh} asserts that Eq.\eqref{eq:20220524-e1} has  no nontrivial  solutions when $\Omega$ is starshaped.\:\,But, as Br\'{e}zis and Nirenberg have shown in \cite{bre2}, a lower-order terms can reverse this circumstance.  Indeed, they considered the following classical problem
    \begin{equation}\label{1.3}
      \begin{cases}
          -\Delta u={\left|u\right|}^{{2}^{\ast }-2}u+\lambda u& \text{ }x\in{\Omega },\\
            \quad \;\:\, u=0&\text{ } x\in{\partial \Omega },
       \end{cases}
    \end{equation}
    with $\lambda\in \R, N\geq 3$ and $\Omega\subset {\mathbb{R}}^{N}$ is a bounded domain.
  They found out that the existence of a solution depends heavily on the values of $\lambda$ and $N$.  Precisely, they showed that:

  \begin{enumerate}
    \item [$(i)$]  when $N \ge 4$ and  $\lambda\in \left(0,\lambda_{1}(\Omega)\right)$, there exists a positive solution for Eq.\eqref{1.3};

    \item  [$(ii)$]   when $N=3$ and $\Omega$ is a ball, Eq.\eqref{1.3} has a positive solution if and only if  $\lambda\in \left(\frac{1}{4}\lambda_{1}(\Omega),\lambda_{1}(\Omega)\right)$;
   \item  [$(iii)$] Eq. \eqref{1.3} has no solutions when $\lambda<0$ and $\Omega$ is starshaped;
  \end{enumerate}
  where $\lambda_1(\Omega)$ denotes the first eigenvalue of $-\Delta $ with zero Dirichlet boundary value. Furthermore, Br\'{e}zis and Nirenberg \cite{bre2} also considered  the following general case:
\begin{equation}\label{bubu1.3}
      \begin{cases}
          -\Delta u={\left|u\right|}^{{2}^{\ast }-2}u+f(x,u)& \text{ }x\in{\Omega },\\
           \quad \;\:\,u=0&\text{ } x\in{\partial \Omega },
       \end{cases}
    \end{equation}
where $f(x,u)$ satisfies some of the following assumptions :

\begin{itemize}
\item[$(f_1)$]$f(x,u)=a(x)u+g(x,u), a(x)\in L^\infty(\Omega);$
\item[$(f_2)$]$\lim\limits_{u\to 0^+} \frac{g(x,u)}{u}=0,$ uniformly in $x\in \Omega$;
\item[$(f_3)$]$\lim\limits_{u\to +\infty} \frac{g(x,u)}{u^{2^*-1}}=0,$ uniformly in $x\in \Omega$;
\item[$(f_4)$]$\exists\,\alpha>0$ such that $\int (|\nabla v|^2-a(x)v^2)dx\geq \alpha\int v^2 dx$ for all $v\in H^1_0(\Omega);$
\item[$(f_5)$]$f(x,u)\geq 0$ for a.e $x\in \omega_0$ and for all $u\geq 0$,  where $\omega_0$ is some nonempty open subset of $\Omega$;
\item[$(f_6)$]$f(x,u)\geq \delta_0>0$ for a.e $x\in \omega_0$ and for all $u\in I$, where $\omega_0$ is given in $(f_5)$, $I\subset (0, +\infty)$ is some nonempty open interval and $\delta_0>0$ is some constant;
\item[$(f_7)$]$f(x,u)\geq \delta_1 u$ for   a.e $x\in \omega_1$ and for all $u\in [0,A],$ or, $f(x,u)\geq \delta_1 u$ for   a.e $x\in \omega_1$ and for all $u\in [A, +\infty],$ where $\omega_1$ is some nonempty open subset of $\Omega$ and $\delta_1, \,A$ are two positive constants;
\item[$(f_8)$]$\lim\limits_{u\to +\infty}\frac{f(x,u)}{u^3}=+\infty$ uniformly in $x\in \omega_2$, where $\omega_2$ is some nonempty open subset of $\Omega$.
\end{itemize}

\noindent They  showed that if the assumptions $(f_1)-(f_4)$ hold and there exists some $0\leq u_0\in H^1_0(\Omega)\setminus\{0\}$ such that $\displaystyle\sup_{t\geq 0}I(tu_0)<\frac{1}{N}S^\frac{N}{2}$, then Eq.\eqref{bubu1.3} has a positive solution. More precisely, they proved that:
\begin{enumerate}
  \item [$(i)$] If $N\geq 5$, Eq.\eqref{bubu1.3} has a positive solution  provided $(f_1)-(f_6)$;
  \item [$(ii)$] If $N=4$, Eq.\eqref{bubu1.3} has a positive solution  provided $(f_1)-(f_5)$ and $(f_7)$;
  \item [$(iii)$]  If $N=3$, Eq.\eqref{bubu1.3} has a positive solution provided $(f_1)-(f_5)$ and $ (f_8)$.
\end{enumerate}
Some similar results can be seen  in \cite{bar,gao,lix}. Barrios et al. \cite{bar} proved the existence of positive solution for a fractional critical problem with a lower-order term, and Gao and Yang \cite{gao},
Li and Ma \cite{lix} considered the existence of positive solution to a Choquard equation with critical exponent and lower-order term in a bounded domain $\Omega$ and in $\R^N$, respectively.

\begin{remark}\label{remark:20220524-r1}
Compared with $|u|^{2^*-2}u$, $u\log u^2$ is a lower-order term at infinity. However,
we note that the situation we considered in present paper is not covered above. Indeed,
in the Eq. \eqref{1.1},
$f(x, u)=\lambda u+\mu u\log {u}^{2}$. So $(f_1)$ fails due to the fact of $\lim\limits_{u\to 0^+}\frac{u\log u^2}{u}=-\infty$. That is $\lambda u=o(\mu u\log {u}^{2})$ for $u$ close to $0$. So it is natural to believe that $\mu u\log u^2$ has much  more influence than $\lambda u$ on the existence of positive solutions to Eq.\eqref{1.1}. Hence, our main goal in present paper is to make clear this guess.
\end{remark}



 To find a positive solution to Eq.\eqref{1.1}, we define a modified functional:
     \begin{equation}\label{fun1}
       I(u)=\frac{1}{2}\int_{\Omega}\left|\nabla u\right|^{2}dx-\frac{1}{2^\ast}\int\left|u_+\right|^{2^\ast}dx
       -\frac{\lambda}{2}\int u_+^2dx-\frac{\mu}{2}\int{u^2_{+}}(\log u^2_{+}-1)dx,~u\in H^1_0(\Omega),
     \end{equation}
     which can be rewritten by
     \begin{equation}\label{fun2}
       I(u)=\frac{1}{2}\int_{\Omega}\left|\nabla u\right|^{2}dx-\frac{1}{2^\ast}\int\left|u_+\right|^{2^\ast}dx
       -\frac{\mu}{2}\int{u^2_{+}}(\log u^2_{+}+\frac{\lambda}{\mu}-1)dx,~u\in H^1_0(\Omega),
     \end{equation}%
     where $u_{+}=\max\{u,0\},\,u_{-}=-\max\{-u,0\}$. It is easy to see that $I$ is well-defined in $H_0^1(\Omega)$ and any nonnegative critical point of $I$ corresponds to a solution of Eq.\eqref{1.1}.

Before stating our results, we introduce some notations.
Hereafter, we use $\int$ to denote $\int_\Omega~\mathrm{d}x$, unless specifically stated,    and let $S$ and $\lambda_1(\Omega)$ be the best Sobolev constant of the embedding  $H^1(\R^N) \hookrightarrow L^{2^*}(\R^N)$ and the first eigenvalue of $-\Delta $ with zero Dirichlet boundary value respectively,  i.e,
$$S:=\inf\limits_{u\in H^1(\R^N)\setminus\{0\}}\frac{\int_{\R^N}|\nabla u|^2~\mathrm{d}x}{(\int_{\R^N}|u|^{2^*}~\mathrm{d}x)^\frac{2}{2^*}}$$
and
$$\lambda_1(\Omega):=\inf\limits_{u\in H^1_0(\Omega)\setminus\{0\}}\frac{\int_{\Omega}|\nabla u|^2~\mathrm{d}x}{\int_{\Omega}|u|^2~\mathrm{d}x}.$$
We also set
$$\left\|v\right\|^2:=\int|\nabla v|^2,\,~~v\in {H}_{0}^{1}(\Omega),$$

$$\mathcal{N}:= \left\{u\in H^1_0(\Omega)\setminus\{0\}\ \ |\ \  g(u)=0\right\},$$

and
\begin{equation}\label{bu1.7}
c_g:=\inf\limits_{u\in \mathcal{N}} I(u),~~~~ c_M:=\inf\limits_{\gamma\in \Gamma}\max\limits_{t\in [0,1]}I(\gamma(t)),
\end{equation}
where
$$g(u):=\int\left|\nabla u\right|^{2}-\int\left|u_+\right|^{2^\ast}
       -\lambda\int u_+^2-\mu\int{u^2_{+}}\log u^2_{+},$$
and
$$\Gamma:=\{\gamma\in C([0,1], H^1_0(\Omega))\ \ |\ \ \gamma(0)= 0, I(\gamma(1))<0\}.$$

Let
\begin{align*}
A_0:=&\left\{(\lambda, \mu)\ \ |\ \  \lambda\in \R, \mu>0\right\},\\
B_0:=&\left\{(\lambda, \mu) \  \ |\ \ \lambda \in [0,\lambda_1(\Omega)),  \mu<0,\frac{1}{N}\left(\frac{\lambda_1(\Omega)-\lambda}{\lambda_1(\Omega)}\right)^\frac{N}{2}S^\frac{N}{2}+\frac{\mu}{2}|\Omega|>0\right\},\\
C_0:=&\left\{ (\lambda, \mu)\ \ |\ \ \lambda \in \R,  \mu<0, \frac{1}{N}S^\frac{N}{2}+\frac{\mu}{2}e^{-\frac{\lambda}{\mu }}|\Omega|>0\right\}.
\end{align*}

Here comes our main results.

\begin{theorem}\label{th1}
If $(\lambda, \mu)\in A_0$ and $N\geq 4$, then  problem \eqref{1.1} has a positive Mountain pass solution, which is also a ground state solution.
\end{theorem}

%

 Denote $f(s):=|s|^{2^*-2}s+\lambda s+\mu s\log s^2$ which is of odd. It is easy to see that $\mathcal{N}\neq \emptyset $ and $c_M\geq c_g$ if problem \eqref{1.1} has a positive mountain pass solution.
On the other hand,  when $\lambda\in \R$ and $ \mu>0$, $\frac{f(s)}{s}$ is strictly increasing in $(0, +\infty)$ and strictly decreasing in $(-\infty, 0)$, which enable one to show that
$c_M\leq c_g$ (See \cite[Theorem 4.2]{wil}). Therefore, the ground state energy $c_g$ equals to the Mountain pass level energy $c_M$, which implies that the mountain pass solution must be a ground state solution. So, in Theorem \ref{th1}, we only need to show that problem \eqref{1.1} has a positive mountain pass solution.

The case of $\mu <0$ is thorny.  Indeed for $(\lambda, \mu)\in B_0\cup C_0$, $I(u)$ still has the mountain pass geometry (See Lemma \ref{Lemma1}). However,  in such a case, it holds that $c_g<c_M$.
Since we can not check the $(PS)_{c_M}$ condition for $I(u)$, we apply the mountain pass theorem without $(PS)_{c_M}$ condition to gain a positive solution for Eq.\eqref{1.1} when $(\lambda, \mu)\in B_0\cup C_0$. However, we don't know whether this solution is of mountain pass type or not.

\begin{theorem}\label{th2}
Problem \eqref{1.1} possesses a positive solution provided one of the following condition holds:
 \begin{itemize}
 \item[(i)]$N=3$, $(\lambda, \mu)\in B_0\cup C_0$;
 \item[(ii)]$N=4$, $(\lambda, \mu)\in B_0\cup C_0$ with $\frac{32 e^\frac{\lambda}{\mu}}{\rho_{max}^2}< 1$,where $\rho_{max}:=\sup\{r>0: \exists x\in \Omega ~s.t. ~B(x,r)\subset \Omega\}$.
 \end{itemize}
\end{theorem}

For the nonexistence of positive solutions for problem \eqref{1.1}, we have the following partial result.
\begin{theorem}\label{buth1}
Assume that $N\geq 3.$ If $\mu<0$ and $-\frac{(N-2)\mu}{2}+\frac{(N-2)\mu}{2}\log(-\frac{(N-2)\mu}{2})+\lambda-\lambda_1(\Omega)\geq 0$, then problem \eqref{1.1}
has no positive solutions.
\end{theorem}

The existence and nonexistence results given by Theorem \ref {th1} - Theorem \ref {buth1} can be described on the $(\lambda, \ \mu)$ plane by Figure \ref {graph1}. The pink regions stand for the existence of positive solution, while the blue regions correspond the non-existence of positive solution. Here $\tau_1$, $\eta_1$, $\eta_2$ and $\eta_3$ are  curves given by
 \begin{align*}
	      & \tau_1: \ -\frac{(N-2)\mu}{2}+\frac{(N-2)\mu}{2}\log(-\frac{(N-2)\mu}{2})+\lambda-\lambda_1(\Omega) = 0,\\
&\eta_1:\ \frac{1}{N}(\frac{\lambda_1(\Omega)-\lambda}{\lambda_1(\Omega)})^\frac{N}{2}S^\frac{N}{2}+\frac{\mu}{2}|\Omega|=0,\\
	     &\eta_2:\frac{1}{N}S^\frac{N}{2}+\frac{\mu}{2}e^{-\frac{\lambda}{\mu }}|\Omega|=0,\\
	      &\eta_3:{32 e^\frac{\lambda}{\mu}}={\rho_{max}^2}, \ \ \ \rho_{max}:=\sup\{r\in (0, +\infty): \exists x\in \Omega ~s.t.~B(x,r)\subset \Omega\}.
	      \end{align*}

\begin{figure}[htbp]\label{graph1}
\centering
\subfigure[$N=3$]{
\begin{minipage}[t]{0.33\linewidth}
\centering
\includegraphics[height=3cm,width=4.9cm]{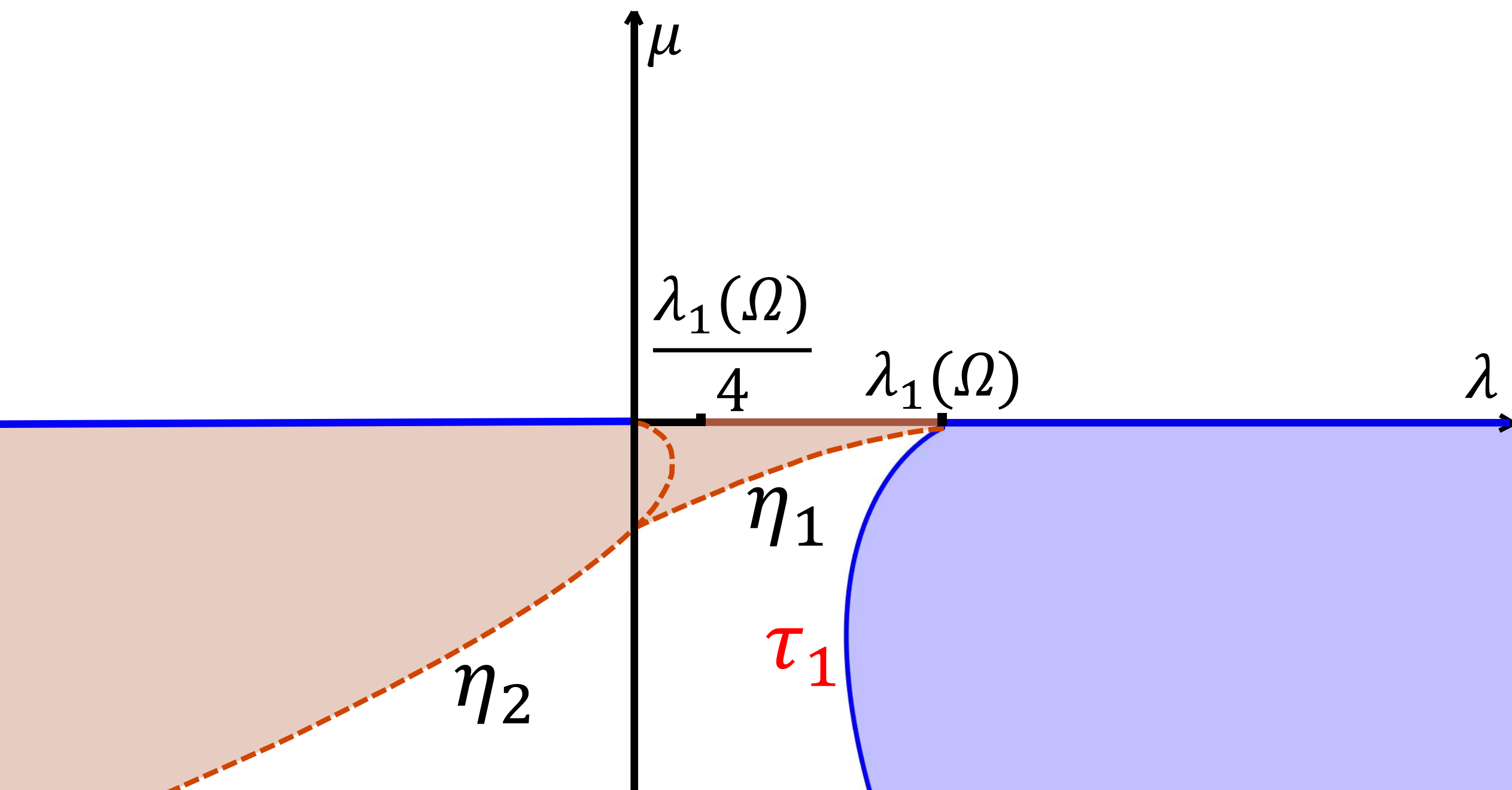}
\end{minipage}%
}%
\subfigure[$N=4$]{
\begin{minipage}[t]{0.33\linewidth}
\centering
\includegraphics[height=3cm,width=4.9cm]{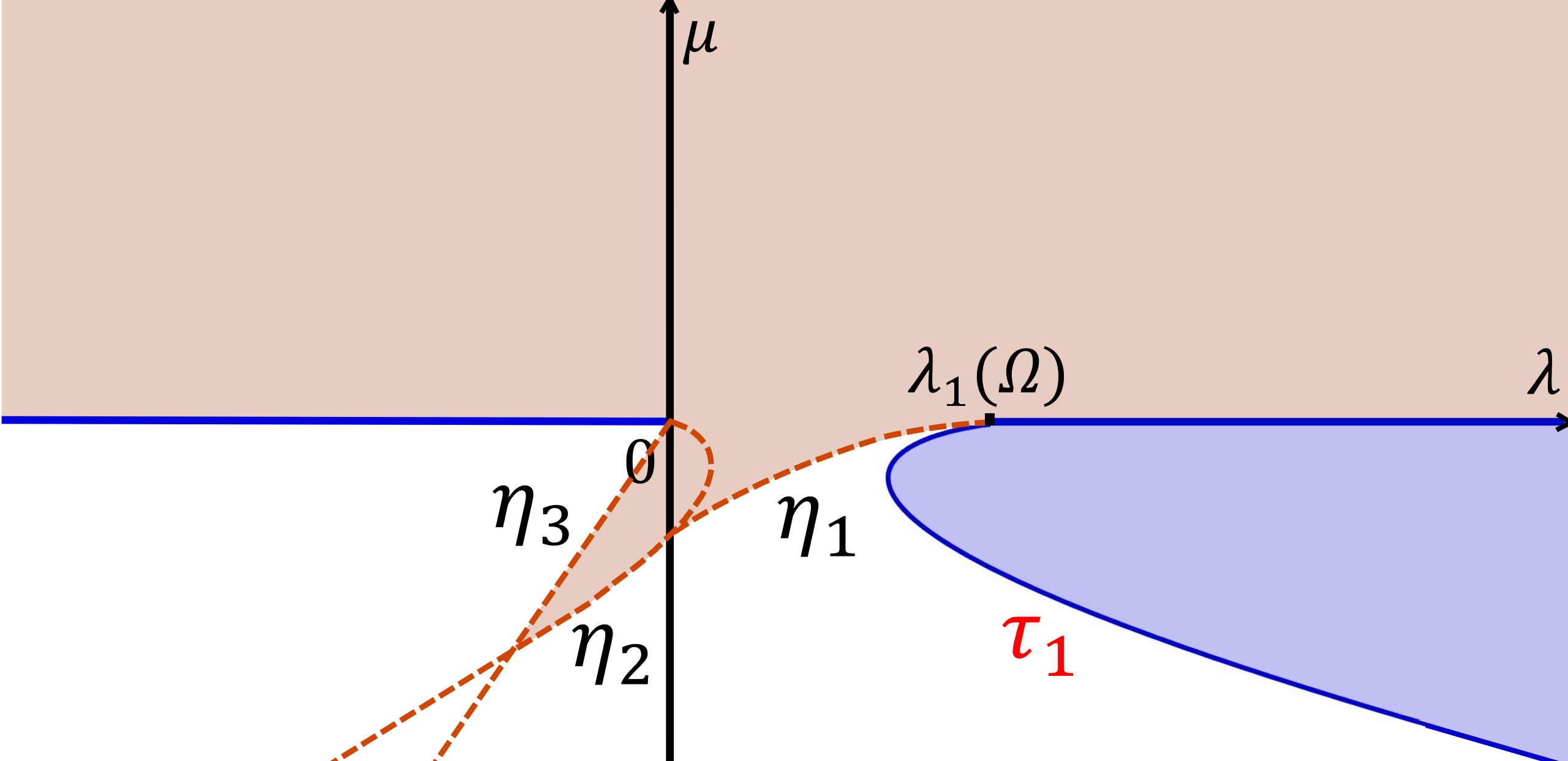}
\end{minipage}%
}%
\subfigure[$N\ge5$]{
\begin{minipage}[t]{0.33\linewidth}
\centering
\includegraphics[height=3cm,width=4.9cm]{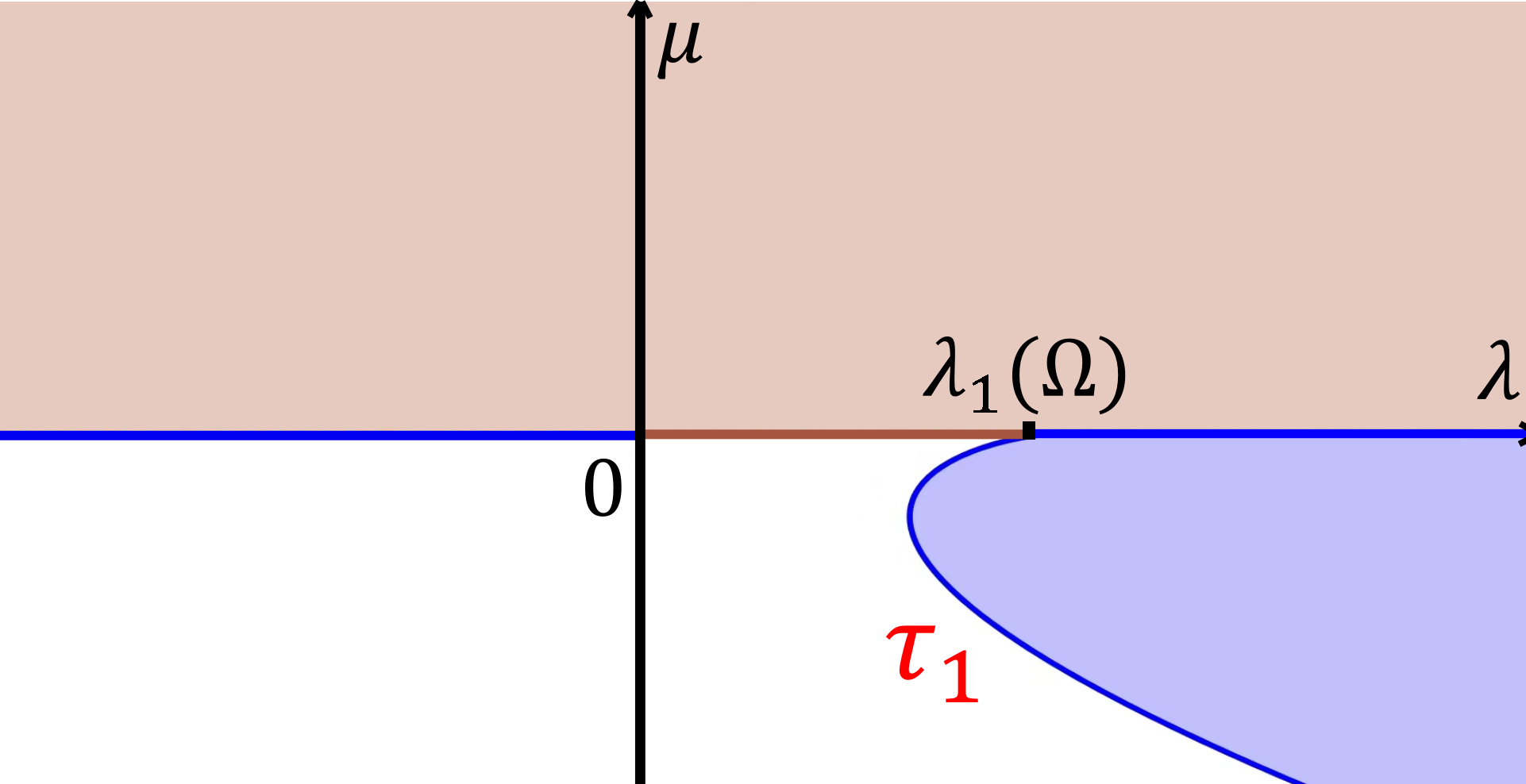}
\end{minipage}
}
\centering
\caption{existence  and  nonexistence}
\end{figure}

\begin{remark}
 Comparing the results of \cite{bre2} and   FIGURE \ref {graph1} above, we find that for the case of $N\geq 4$: Eq.\eqref{1.1} possesses a positive solution only for $\lambda\in (0, \lambda_1(\Omega))$ if $\mu =0$. while it has a positive solution for all $\lambda\in \R$ if $\mu>0$. So we see that $\mu u\log u^2\ (\mu>0)$ really plays a leading role (compared with $\lambda u$) in the effect on the existence of positive solution to  Eq.\eqref{1.1}. A similar phenomenon occurs  for $N=3$: Eq.\eqref{1.1} has a positive solution  only for $\lambda\in (\lambda^*,\  \lambda_1(\Omega))\subset(0, \ \lambda_1(\Omega)) $ if $\mu=0$, while it has a positive solution for all $\lambda\in (-\infty,\lambda_1(\Omega))$ if $\mu<0$.
\end{remark}

   Before closing the introduction, we give the outline of our paper. In Section \ref{section:preliminary}, we will check the mountain pass geometry structure for $I(u)$, under different specific situations. We also give some other preliminaries.
    In  Section \ref{section:Estimations},  we are devoted to estimate the mountain pass level $c_M$ for different parameters $\lambda,\mu$ and $N$. The proofs of our main Theorems \ref{th1}, \ref{th2} and \ref{buth1} are given  in  Section \ref{section:proofs}.


   \section{Preliminaries}\label{section:preliminary}


   In this section,   firstly  we  verify the mountain pass geometry structure for $I(u)$ when $(\lambda,\mu)\in A_0\cup B_0\cup C_0$.  Secondly we show that $I(u)$ satisfies $(PS)_d$ condition provided $d<\frac {1}{N}S^{\frac {N}{2}}$. Finally, we deduce a existence result for Eq.(\ref{1.1}) when $c_M\in (-\infty, \ 0)\cup (0, \ \frac {1}{N}S^{\frac {N}{2}})$ and $(\lambda,\mu)\in B_0\cup C_0$.

    \begin{lemma}\label{Lemma1}
	  Assume that $N\geq 3$ and    $(\lambda,\mu)\in A_0\cup B_0\cup C_0$. Then the functional $I(u)$ satisfies the mountain pass geometry structure:	
\begin{itemize}
\item[$(i)$]there exist $\alpha, \rho>0$ such that $I(v)\ge\alpha$ for all $\left\|v\right\|=\rho$;
\item[$(ii)$]there exists $\omega\in H_{0}^{1}(\Omega)$ such that $\left\|\omega\right\|\ge\rho$ and $I(\omega)<0$.
\end{itemize}
    \end{lemma}

	\begin{proof}
We divide the proof into three  cases.

\textbf{Case 1:} $(\lambda,\mu)\in A_0$.

Since    $\mu> 0$, it follows from the fact  $s^2log s^2\leq C s^{2^*}$ for all $s\in [1,+\infty)$ that
	   \begin{align*}
	      &\mu\int u^2_{+}(\log u^2_{+}+\frac{\lambda}{\mu}-1)
	      \le \mu\int u^2_{+}(\log u^2_{+}+\frac{\lambda}{\mu})
=\mu\int u^2_{+}\log (e^\frac{\lambda}{\mu}u^2_{+})\\
	      =&\mu\int _{\{e^\frac{\lambda}{\mu}u^2_{+}\ge 1\}}u^2_{+}\log (e^\frac{\lambda}{\mu}u^2_{+}) +\mu\int _{\{e^\frac{\lambda}{\mu}u^2_{+}\le1 \}}u^2_{+}\log (e^\frac{\lambda}{\mu}u^2_{+})\\
	      \le&\mu\int _{\{e^\frac{\lambda}{\mu}u^2_{+}\ge 1\}}u^2_{+}\log (e^\frac{\lambda}{\mu}u^2_{+})
	      \le C\mu\int _{\{e^\frac{\lambda}{\mu}u^2_{+}\ge 1\}}e^\frac{(2^*-2)\lambda}{2\mu}\left|u_{+}\right|^{2^\ast}\\
	      \le& C e^\frac{(2^*-2)\lambda}{2\mu}\mu\int\left|u\right|^{2^\ast}
	      \le Ce^\frac{(2^*-2)\lambda}{2\mu}\mu\left\|u\right\|^{2^\ast}.
       \end{align*}
      So
       \begin{align*}
          I(u)\ge\frac{1}{2}\left\|u\right\|^2-C_1\left\|u\right\|^{2\ast}
          -C_2\left\|u\right\|^{2^\ast} ~\hbox{for some $C_1,C_2>0$},
       \end{align*}
       which implies that there exist  $\alpha>0$ and $\rho>0$ such that
       $I(v)\ge\alpha>0$ for all $\left\|v\right\|=\rho$.

	   Let $0\leq\varphi \in H_{0}^{1}(\Omega)\setminus\{0\}$ be a fixed function, then
	   \begin{align*}  	      I(t\varphi)&=\frac{t^2}{2}\int\left|\nabla\varphi\right|^2-\frac{\left|t\right|^{2^\ast}}{2^\ast}\int\left|\varphi\right|^{2^\ast}-\frac{\mu}{2}t^2\int
\varphi^2(\log(t^2\varphi^2)+\frac{\lambda}{\mu}-1)\\
          &=\frac{t^2}{2}\int\left|\nabla\varphi\right|^2-\frac{\left|t\right|^{2^\ast}}{2^\ast}\int\left|\varphi\right|^{2^\ast}-\frac{\mu}{2}t^2\int
          \varphi^2(\log t^2+\log\varphi^2+\frac{\lambda}{\mu}-1)\\
          &=\frac{t^2}{2}\int\left|\nabla\varphi\right|^2-\frac{\left|t\right|^{2^\ast}}{2^\ast}\int\left|\varphi\right|^{2^\ast}-\frac{\mu}{2}t^2\log t^2\int\varphi^2-\frac{\mu}{2}t^2\int\varphi^2(\log\varphi^2+\frac{\lambda}{\mu}-1)\\
          &\to-\infty\quad as\quad  t\to+\infty,
       \end{align*}
       since $\lim\limits_{t\to +\infty}\frac{t^{2^\ast}}{t^2\log t^2}=+\infty$.
       Therefore, we can choose $t_{0}\in \R^+$ large enough such that
       \begin{align*}
       	   I(t_{0}\varphi)<0 \quad and\quad \left\|t_{0}\varphi \right\|>\rho .
       \end{align*}

     \textbf{ Case 2:} $(\lambda,\mu)\in B_0$.

Since  $\mu<0$, we have
	   \begin{align*}
	      &-\frac{\mu}{2}\int u^2_{+}(\log u^2_{+}-1)
=-\frac{\mu}{2}\int u^2_{+}\log(e^{-1} u^2_{+})\\
	      =&-\frac{\mu}{2}\int _{\{e^{-1} u^2_{+}\geq 1\}}u^2_{+}\log (e^{-1} u^2_{+}) -\frac{\mu}{2}\int _{\{e^{-1} u^2_{+}\leq 1 \}}u^2_{+}\log (e^{-1} u^2_{+}) \\
	      \ge&-\frac{\mu}{2}\int _{\{e^{-1} u^2_{+}\leq 1 \}}u^2_{+}\log (e^{-1} u^2_{+})\\
	      \ge& -\frac{\mu}{2}e\int _{\{e^{-1} u^2_{+}\leq 1 \}}-e^{-1}~ \mathrm{d}x  \ge \frac{\mu}{2}|\Omega|.
       \end{align*}
       It follows that
        \begin{equation}\label{bu5.71}
       I(u)\geq \frac{1}{2}\frac{\lambda_1(\Omega)-\lambda}{\lambda_1(\Omega)}\int |\nabla u|^2-\frac{1}{2^*}S^{-\frac{2^*}{2}}\left(\int |\nabla u|^2\right)^\frac{2^*}{2}
       +\frac{\mu}{2}|\Omega|.
              \end{equation}
 Put  $\alpha :=\frac{1}{N}\left(\frac{\lambda_1(\Omega)-\lambda}{\lambda_1(\Omega)}\right)^\frac{N}{2}S^\frac{N}{2}+\frac{\mu}{2}|\Omega|$ and $\rho :=\left(\frac{\lambda_1(\Omega)-\lambda}{\lambda_1(\Omega)}\right)^\frac{N-2}{4}S^\frac{N}{4}$, then $ \alpha >0$ and $\rho >0$ due to the fact $(\lambda, \mu)\in B_0$.  By (\ref {bu5.71}),
       $$I(v)\geq \frac{1}{2}\frac{\lambda_1(\Omega)-\lambda}{\lambda_1(\Omega)}\rho ^2-\frac{1}{2^*}S^{-\frac{2^*}{2}}\rho ^{2^*}
       +\frac{\mu}{2}|\Omega|=\frac{1}{N}\left(\frac{\lambda_1(\Omega)-\lambda}{\lambda_1(\Omega)}\right)^\frac{N}{2}S^\frac{N}{2}+\frac{\mu}{2}|\Omega|=\alpha >0$$
      for any $||v||=\rho $.
       Applying a similar argument as Case 1 above, we can find a
       function $w\in H^1_0(\Omega)$ such that  $\left\|\omega\right\|$$\ge$$\rho$ and $I(\omega)<0$.

\textbf{Case 3:} $(\lambda,\mu)\in C_0$.

Since  $\mu<0$, we have
	   \begin{align*}
	     & -\frac{\lambda}{2}\int u_{+}^2-\frac{\mu}{2}\int u^2_{+}\left(\log u^2_{+}-1\right)
  =-\frac{\mu}{2}\int u^2_{+}\left(\log u^2_{+}+\frac{\lambda}{\mu }-1\right)\\
&=-\frac{\mu}{2}\int u^2_{+}\log(e^{\frac{\lambda}{\mu }-1} u^2_{+})\\
	      &=-\frac{\mu}{2}\int _{\{e^{\frac{\lambda}{\mu }-1} u^2_{+}\geq 1\}}u^2_{+}\log \left(e^{\frac{\lambda}{\mu }-1} u^2_{+}\right) -\frac{\mu}{2}\int _{\{e^{\frac{\lambda}{\mu }-1} u^2_{+}\leq 1 \}}u^2_{+}\log \left(e^{\frac{\lambda}{\mu }-1} u^2_{+}\right) \\
	      &\ge-\frac{\mu}{2}\int _{\{e^{\frac{\lambda}{\mu }-1} u^2_{+}\leq 1 \}}u^2_{+}\log \left(e^{\frac{\lambda}{\mu }-1} u^2_{+}\right)\\
	      &\ge -\frac{\mu}{2}e^{1-\frac{\lambda}{\mu }}\int _{\{e^{\frac{\lambda}{\mu }-1} u^2_{+}\leq 1 \}}-e^{-1}~ \mathrm{d}x
	      \ge \frac{\mu}{2}e^{-\frac{\lambda}{\mu }}|\Omega|.
       \end{align*}
       It follows that
        \begin{equation}\label{bubu5.71}
       I(u)\geq \frac{1}{2}\int |\nabla u|^2-\frac{1}{2^*}S^{-\frac{2^*}{2}}\left(\int |\nabla u|^2\right)^\frac{2^*}{2}
       +\frac{\mu}{2}e^{-\frac{\lambda}{\mu }}|\Omega|.
              \end{equation}
 Put  $\alpha:=\frac{1}{N}S^\frac{N}{2}+\frac{\mu}{2}e^{-\frac{\lambda}{\mu }}|\Omega|$ and $ \rho :=S^\frac{N}{4}$, then $ \alpha>0$ due to that $(\lambda, \mu)\in C_0$. By \eqref{bubu5.71},
       $$I(v)\geq \frac{1}{2} \rho^2-\frac{1}{2^*}S^{-\frac{2^*}{2}}\rho^{2^*}
       +\frac{\mu}{2}e^{-\frac{\lambda}{\mu }}|\Omega|=\frac{1}{N}S^\frac{N}{2}+\frac{\mu}{2}e^{-\frac{\lambda}{\mu }}|\Omega|=\alpha>0$$
        for any $||v||=\rho$.
       Similarly, it is not hard to find a
       function $w\in H^1_0(\Omega)$ such that  $\left\|\omega\right\|$$\ge$$\rho$ and $I(\omega)<0$.
  	\end{proof}


    \begin{lemma}\label{lm2.3} Assume that $N\geq 3, \lambda\in \R$ and $\mu\in \R\setminus\{0\}$.  Then
	   any  $(PS)_{d}$ sequence $\{u_{n}\}$ of $I$ must be    bounded   in $H^1_{0}(\Omega)$ for all $d\in \R$.
    \end{lemma}

    \begin{proof}
       By the definition of the  $(PS)_{d} $ sequence, we have that, as $n \to +\infty,$
       $$I(u_{n}) \to d~~\hbox{and}~~I^{'}(u_{n}) \to 0~\hbox{in}~H^{-1}(\Omega).$$
That is,
\begin{equation}\label{1.4}
\begin{split}
     	 \frac{1}{2}\int \left|\nabla u_{n}\right|^2-\frac{1}{2^\ast}\int\left|{(u_{n})}_{+}\right|^{2^\ast}-\frac{\mu}{2}\int(u_{n})^2_{+}\log(u_{n})^2_{+}+\frac{\mu-\lambda}{2}\int(u_{n})^2_{+}
     	  =d+o_{n}(1),
\end{split}
\end{equation}
       and
        \begin{equation}\label{1.5}
        \begin{split}
     	  \int \left|\nabla u_{n}\right|^2-\int\left|{(u_{n})}_{+}\right|^{2^\ast}-\lambda\int(u_{n})^2_{+}-\mu\int(u_{n})^2_{+}\log(u_{n})^2_{+}
     	  =o_{n}(1)\left\|u_{n}\right\|
       \end{split}
       \end{equation}
as $n \to +\infty.$ Now we divide the proof into two cases.

\textbf{Case 1:}  $\mu>0 $.

 It follows from \eqref{1.4} and \eqref{1.5} that
       \begin{equation*}
       \begin{split}
     	  &d+o_{n}(1)+o_{n}(1)\left\|u_{n}\right\|=I(u_{n})-\frac{1}{2}\langle I^{'}(u_{n}),u_{n}\rangle\\
     	  =&\frac{1}{N}\int\left|(u_{n})_{+}\right|^{2^\ast}+\frac{\mu}{2}\int(u_{n})^2_{+}
     	  \ge \frac{\mu}{2}\int(u_{n})^2_{+},
       \end{split}
       \end{equation*}	
     	thus $\left|(u_{n})_{+}\right|^2_{2}\le C+C\left\|u_{n}\right\|$.
        Using $\eqref{1.4}$ and $\eqref{1.5}$ again, we have, for $n$ large enough,
        \begin{equation*}
        \begin{split}
        	&2d+\left\|u_{n}\right\|\ge I(u_{n})-\frac{1}{2^\ast}\langle I^{'}(u_{n}),u_{n}\rangle\\
        	 =&\frac{1}{N}\left\|u_{n}\right\|^2-\frac{\lambda}{N}\int(u_{n})^2_{+}+\frac{\mu}{2}\int(u_{n})^2_{+}-\frac{1}{N}\mu\int(u_{n})^2_{+}\log(u_{n})^2_{+}.
        \end{split}
        \end{equation*}
      Recalling the following inequality (see \cite{shu} or see \cite[Theorem 8.14]{lie2})
       \begin{equation*}
       	\int u^2\log u^2\le \frac{a}{\pi}\left\|u\right\|^2+(\log|u|^2_{2}-N(1+\log a))|u|^2_{2}  ~\hbox{for} ~u\in H^1_{0}(\Omega) ~ \hbox{and}\: a>0,
       \end{equation*}
  we have that
      \begin{equation*}
          \begin{split}
     		\frac{1}{N}\left\|u_{n}\right\|^2 &\le 2d+\left\|u_{n}\right\|+C\int(u_{n})^2_{+}+\frac{1}{N}\mu\int(u_{n})^2_{+}\log(u_{n})^2_{+}\\
     		&\le C+C\left\|u_{n}\right\|+\frac{1}{N}\mu\left[\frac{a}{\pi}\left\|u_{n}\right\|^2+\left(\log\left|(u_{n})_{+}\right|^2_{2}-N(1+\log a)\right)\left|(u_{n})_{+}\right|^2_{2}\right]\\
     		&\le C+C\left\|u_{n}\right\|+\frac{1}{2N}\left\|u_{n}\right\|^2+\left|\left|(u_{n})_{+}\right|^2_{2}\log \left|(u_{n})_{+}\right|^2_{2}\right|+C\left|(u_{n})_{+}\right|^2_{2}\\
     		&\le C+C\left\|u_{n}\right\|+\frac{1}{2N}{\lVert u_{n}\rVert}^2+C\left|(u_{n})_{+}\right|^{2-\delta}_{2}+C\left|(u_{n})_{+}\right|^{2+\delta}_{2}+C\left|(u_{n})_{+}\right|^{2}_{2}\\
     		&\le C+C\left\|u_{n}\right\|+\frac{1}{2N}{\lVert u_{n}\rVert}^2+C\left(C+C\left\|u_{n}\right\|\right)^{\frac{2-\delta}{2}}+C\left(C+C\left\|u_{n}\right\|\right)^{\frac{2+\delta}{2}},
     	\end{split}	
      \end{equation*}
     where  $a>0$ with $\frac{a}{\pi}\mu<\frac{1}{2}$ and $\delta\in (0,1)$.
     So  there exists $C>0$ such that $ \| u_{n}\|<C.$

  \textbf{Case 2:} $\mu<0$.

  For $n$ large enough, we have
 \begin{equation}\label{bubu1.8}
 \begin{split}
        	2d+\left\|u_{n}\right\|&\geq \frac{1}{N}\left\|u_{n}\right\|^2-\frac{\lambda}{N}\int(u_{n})^2_{+}+\frac{\mu}{2}\int(u_{n})_+^2-\frac{1}{N}\mu\int(u_{n})^2_{+}\log(u_{n})^2_{+}\\
        &=\frac{1}{N}\left\|u_{n}\right\|^2-\frac{1}{N}\mu\int(u_{n})^2_{+}\log(e^{\frac{\lambda}{\mu}-\frac{N}{2}}(u_{n})^2_{+})\\
        &\geq\frac{1}{N}\left\|u_{n}\right\|^2-\frac{1}{N}\mu\int_{\{e^{\frac{\lambda}{\mu}-\frac{N}{2}}(u_{n})^2_{+}\leq1\}}(u_{n})^2_{+}\log(e^{\frac{\lambda}{\mu}-\frac{N}{2}}(u_{n})^2_{+})\\
        &\geq \frac{1}{N}\left\|u_{n}\right\|^2-\frac{1}{N}\mu\int_{\{e^{\frac{\lambda}{\mu}-\frac{N}{2}}(u_{n})^2_{+}\leq1\}} -e^{\frac{N}{2}-\frac{\lambda}{\mu}-1}~\mathrm{d}x\\
        &\geq \frac{1}{N}\left\|u_{n}\right\|^2+\frac{\mu}{N} e^{\frac{N}{2}-\frac{\lambda}{\mu}-1}|\Omega|,\\
        \end{split}
        \end{equation}
which implies that $\{u_n\}$ is bounded in $H_0^1(\Omega).$
 \end{proof}

 \begin{lemma}\label{lm2.32}
 	Let $\{u_{n}\}$ be a bounded sequence in $H^1_{0}(\Omega)$ such that
 	$u_{n}\to u $ a.e in $ \Omega$ as $n\rightarrow \infty$, then
 	
 		\begin{equation}\label{1.6}
 \lim\limits_{n\to \infty}\int_{\Omega}u^2_{n}\log u^2_{n}\mathrm{d}x=\int_{\Omega}u^2\log u^2\mathrm{d}x,
 \end{equation}
 and
 		
 \begin{equation}\label{1.7}
 \lim\limits_{n\to\infty}\int_{\Omega}(u_{n})^2_{+}\log(u_{n})^2_{+}\mathrm{d}x=\int_{\Omega}u^2_{+}\log u^2_{+}\mathrm{d}x.
 \end{equation}
 \end{lemma}

     \begin{proof}
     	We only prove $\eqref{1.6}$.  And  $\eqref{1.7}$ can be proved similarly.\\
     	Under the conditions, there exists some $C>0$ such that
     	\begin{gather*}
     		\left|\int_{\Omega}u^2_{n}\log u^2_{n}\right|\le C~ \hbox{and}~\left|\int_\Omega u^2\log u^2\right|\le C
     	\end{gather*}
     	By \cite[Lemma 3.1]{shu}, we have
     	\begin{gather*}
     		\lim\limits_{n\to\infty}\int_{\Omega}u^2_{n}\log u^2_{n}-\left|u_{n}-u\right|^2\log \left|u_{n}-u\right|^2=\int_{\Omega} u^2\log u^2.
     	\end{gather*}
     	Since $\left|s^2\log s^2\right|\le Cs^{2-\delta}+Cs^{2+\delta}$ and the embedding of  $H^1_{0}(\Omega)\hookrightarrow L^p(1\le p<2^\ast)$ is compact, we obtain that
     	\begin{gather*}
     		\left|\int_{\Omega} \left|u_{n}-u\right|^2\log \left|u_{n}-u\right|^2 \mathrm{d}x\right|\le C\int_{\Omega}\left|u_{n}-u\right|^{2-\delta}+C\int_{\Omega}\left|u_{n}-u\right|^{2+\delta}\to 0\quad as \quad n\to\infty.\\
     	     	\end{gather*}
     Hence,
     $$\lim\limits_{n\to\infty}\int_{\Omega}u^2_{n}\log u^2_{n}=\int_{\Omega}u^2\log u^2.$$
     	
     \end{proof}
 \begin{lemma}\label{lm2.4}
 	If $N\geq 3, \lambda\in \R,  \mu>0$ and  $d<\frac{1}{N}S^{\frac{N}{2}}$, then $I(u)$ satisfies the $(PS)_d$ condition.
 \end{lemma}
     \begin{proof}
     	Let $\{u_{n}\}$ be a $(PS)_{d}$ sequence of $I$. By Lemma \ref{lm2.3}, we know that $\{u_{n}\}$ is bounded in $H^1_{0}(\Omega)$.  So there exists $u\in H^1_{0}(\Omega)$ such that, up to a subsequence,
     	$$\begin{array}{ll}
     		u_{n}&\rightharpoonup u\quad in\quad H^1_{0}(\Omega),\\
     		u_{n}&\rightarrow u\quad in\quad L^q(\Omega),\quad 1\le q<2^\ast,\\
     		u_{n}&\rightarrow u\quad a.e \quad in \quad \Omega.
     	\end{array}$$
     	Since $\langle I^{'}(u_{n}),\varphi \rangle \to0$ as $n\to\infty$ for any $\varphi\in C^{\infty}_{0}(\Omega)$, $u$ is a weak solution to
     	\begin{gather*}
     		-\Delta u=\left|u_{+}\right|^{2^\ast-2}u_{+}+\lambda u_++\mu u_{+}\log u^2_{+},
     	\end{gather*}
     	which implies that
     	
     		$$\int\left|\nabla u\right|^2=\int\left|u_{+}\right|^{2^\ast}+\lambda \int u^2_{+}+\mu\int u^2_{+}\log u^2_{+}$$
     		and
   \begin{equation}\label{2.5}
    \begin{array}{ll}
   I(u)&=\frac{1}{2}  \ds \int\left|\nabla u\right|^2-\frac{1}{2^\ast}  \ds \int\left|u_{+}\right|^{2^\ast}-\frac{\lambda}{2}\int u^2_{+}-\frac{\mu}{2}  \ds \int u^2_{+}(\log u^2_{+}-1)\\
     		&=\frac{1}{N}  \ds \int\left|u_{+}\right|^{2^\ast}+\frac{\mu}{2}  \ds \int u^2_{+}\ge 0.
     	\end{array}
     \end{equation}
     	Following from the definition of $(PS)_{d}$ sequence, we have
     	$$
     		\int\left|\nabla u_{n}\right|^2-\int\left|(u_{n})_{+}\right|^{2^\ast}-\lambda\int (u_{n})^2_{+}-\mu\int (u_{n})^2_{+}\log (u_{n})^2_{+}=o_{n}(1)$$
     and
     $$		\frac{1}{2}\int\left|\nabla u_{n}\right|^2-\frac{1}{2^\ast}\int\left|(u_{n})_{+}\right|^{2^\ast}-\frac{\lambda}{2}\int (u_{n})^2_{+}-\frac{\mu}{2}\int (u_{n})^2_{+}(\log (u_{n})^2_{+}-1)=d+o_{n}(1).
 		$$
 		Set   $v_{n}=u_{n}-u$. Then
     	$$
     		\int\left|\nabla v_{n}\right|^2-\int\left|(v_{n})_{+}\right|^{2^\ast}=o_{n}(1)$$
     and
     		$$I(u)+\frac{1}{2}\int\left|\nabla v_{n}\right|^2-\frac{1}{2^\ast}\int\left|(v_{n})_{+}\right|^{2^\ast}=d+o_{n}(1).$$
     	Let
     	$$
     	    \ds\int\left|\nabla v_{n}\right|^2\to k, ~ \hbox{as} ~ n\to \infty.
     	$$
     	So
     	$$
     	    \int\left|(v_{n})_{+}\right|^{2^\ast}\to k, ~ \hbox{as} ~  n\to\infty.
     	 $$
     	By the definition of $S$, we have
     	$$      		\left|\nabla u\right|^2_{2} \ge S\left|u\right|^2_{2^\ast},~~\forall u\in H^1_{0}(\Omega)$$
     	and
     	$$k+o_{n}(1)=\int\left|\nabla v_{n}\right|^2 \ge S(\int\left|(v_{n})_{+}\right|^{2^\ast})^{\frac{2}{2^\ast}}=Sk^\frac{N-2}{N}+o_{n}(1).
     	$$
     	If $k> 0$, then $k\ge S^\frac{N}{2}$.
     	By  $\eqref{2.5}$, we have
     	$$\begin{array}{ll}
     		0\le I(u)\le d-(\frac{1}{2}-\frac{1}{2^\ast})k\le d-\frac{1}{N}S^{\frac{N}{2}}<0,
     	\end{array}$$
     	which is impossible.
     	So $k=0$ and thus
     	$$\begin{array}{ll}
     	u_{n}\rightarrow u,~\hbox{ in} ~ H^1_{0}(\Omega)
     	\end{array}$$
     	
     \end{proof}

\begin{lemma}\label{lm2.41}
 	Assume that $N\geq 3,  \lambda\in \R, \mu<0$ and $c\in (-\infty, 0)\cup (0, \frac{1}{N}S^{\frac{N}{2}})$. If $\{u_n\}$ is a $(PS)_c$ sequence of $I$, then there exists a $u\in H^1_0(\Omega)\setminus\{0\}$ such that $u_n \rightharpoonup u$ weakly in $H^1_0(\Omega)$ and   $u$ is a nonnegative  weak solution of \eqref{1.1}.
 \end{lemma}
     \begin{proof}
     	Let $\{u_{n}\}$ be a $(PS)_{c}$ sequence of $I$. By Lemma \ref{lm2.3}, we know that $\{u_{n}\}$ is bounded in $H^1_{0}(\Omega)$.  So there exists $u\in H^1_{0}(\Omega)$ such that, up to a subsequence,
     	$$\begin{array}{ll}
     		u_{n}&\rightharpoonup u\quad in\quad H^1_{0}(\Omega),\\
     		u_{n}&\rightarrow u\quad in\quad L^q(\Omega),\quad 1\le q<2^\ast,\\
     		u_{n}&\rightarrow u\quad a.e \quad in \quad \Omega.
     	\end{array}$$
     	Since $\langle I^{'}(u_{n}),\varphi \rangle\to0$ as $n\to\infty$ for any $\varphi\in C^{\infty}_{0}(\Omega)$, $u$ is a weak solution to
     	\begin{equation}\label{1227.1}
     		-\Delta u=\left|u_{+}\right|^{2^\ast-2}u_{+}+\lambda u_++\mu u_{+}\log u^2_{+}.
     	\end{equation}
%
    Assume that $u=0$ and set $v_{n}:=u_{n}-u$. 	Following from the definition of $(PS)_{c}$ sequence and Brezis-Lieb Lemma, we have
     	$$
     		\int\left|\nabla v_{n}\right|^2-\int\left|(v_{n})_{+}\right|^{2^\ast}=o_{n}(1)$$
     and
     		\begin{equation}\label{1228.1}
     \frac{1}{2}\int\left|\nabla v_{n}\right|^2-\frac{1}{2^\ast}\int\left|(v_{n})_{+}\right|^{2^\ast}=c+o_{n}(1).
     \end{equation}
     	Let
     	$$
     	    \ds\int\left|\nabla v_{n}\right|^2\to k, ~ \hbox{as} ~ n\to \infty.
     	$$Then
     	$$
     	    \int\left|(v_{n})_{+}\right|^{2^\ast}\to k, ~ \hbox{as} ~  n\to\infty.
     	 $$
     It is easy to see that $k>0$. In fact, if $k=0$, then $\int\left|\nabla u_{n}\right|^2=\int\left|\nabla v_{n}\right|^2\to 0$, which implies that $I(u_n)\to 0$, contradicting to $c\neq 0$.
     Going on as Lemma \ref{lm2.4}, we can obtain that
    $k\ge S^\frac{N}{2}$. So,
     	by  $\eqref{1228.1}$, we have
     	$$\begin{array}{ll}
     		\frac{1}{N}S^\frac{N}{2}\leq \frac{1}{N}k=(\frac{1}{2}-\frac{1}{2^\ast})k=c<\frac{1}{N}S^{\frac{N}{2}},
     	\end{array}$$
     	a contradiction.
     	Hence, $u\neq 0$.

     By the density of $C_0^\infty$ in $H^1_0(\Omega)$ and \eqref{1227.1}, we have that
    $$ \int\left|\nabla u_-\right|^2=0, $$
    which implies that $u\geq 0$.
    Therefore, we can see that $u\in H^1_0(\Omega)\setminus\{0\}$  and $u$ is a nonnegative  weak solution of \eqref{1.1}.

     \end{proof}


    \section{Estimations on $c_M$ }\label{section:Estimations}

    In this section, we are going to give an estimation that $c_M<\frac 1N S^{\frac N2}$, under different assumptions on parameters $\lambda$, $\mu$ and dimension $N$.  Inspired by Br\'{e}zis-Nirengberg\cite{bre2}, it is sufficient to find some suitable $U_\epsilon\in H^1_0(\Omega)$ such that $\sup_{t\ge0}I(tU_\epsilon)<\frac{1}{N}S^{\frac{N}{2}}$.  Without loss of generality, we may assume that $0\in \Omega$, in particular, we suppose that $0$ is the geometric center of $\Omega$, i.e., $\rho_{\max}=dist(0,\partial\Omega)$.

    It is well-known  (see \cite{edm,guy,van})  that the following problem
     \begin{gather*}
        \begin{cases}
           -\Delta u=\left|u\right|^{2^\ast-2}u ,& x\in\mathbb{R}^N,\\
          \quad\:\,\,u>0,&\\
          \:\,u(0)=\max\limits_{x\in\mathbb{R}^N} u(x),&\\
        \end{cases}
     \end{gather*}
      has a unique solution $\widetilde{u}(x)$
     \begin{equation*}
     	 {\widetilde{u}}(x)=\left[N(N-2)\right]^{\frac{N-2}{4}}\frac{1}{{(1+\left|x\right|^2)}^{\frac{N-2}{2}}}.
     \end{equation*}
     And correspondingly, up to a dilations,
     \begin{equation*}
     u_{\epsilon}(x)=\left[N(N-2)\right]^{\frac{N-2}{4}}
     \left(\frac{\epsilon}{\epsilon^2+\left|x\right|^2}\right)^{\frac{N-2}{2}}\\
     \end{equation*}
     is a minimizer for $S$.

    We let $\varphi(x)\in C^\infty_{0}(\Omega)$ be  such that $\varphi(x)\equiv 1$ for $x$ in some neighborhood  $B_\rho(0)$ of $0$, and define
     \begin{equation}
     U_{\epsilon}(x)=\varphi(x)u_{\epsilon}(x).
    \end{equation}
  \begin{lemma}
  	If $N\ge4$, then we have, as $\epsilon\to 0^+$,
 \begin{equation}
  	   \int_\Omega\left|\nabla U_\epsilon\right|^2=S^{\frac{N}{2}}+O(\epsilon^{N-2}),
  \end{equation}
  \begin{equation}
  	   \int_{\Omega}\left|U_{\epsilon}\right|^{2^\ast}=S^{\frac{N}{2}}+O(\epsilon^N),
  \end{equation}
   and
   \begin{gather*}
   	  \int_{\Omega}\left|U_{\epsilon}\right|^2=
   	  \begin{cases}
  	     d\epsilon^2\left|\ln\epsilon\right|+O(\epsilon^2),~&if~ N=4,\\ d\epsilon^2+O(\epsilon^{N-2}),~&if~ N\ge5,
      \end{cases}
   \end{gather*}
   where $d$ is a positive constant.
  \end{lemma}
  \begin{proof}
  	The proof can be found in \cite{wil}.
  \end{proof}
  \begin{lemma}
  	If $N\ge5$, then we have,  as $\epsilon\to0^+$,
  	\begin{equation*}
  		\int_{\Omega}U^2_{\epsilon}\log U^2_{\epsilon}=C_0\epsilon^2\log\frac{1}{\epsilon}+O({\epsilon}^2),
  	\end{equation*}
  where $C_0$ is a positive constant.
  \end{lemma}
  \begin{proof}
  	 $$\begin{array}{ll}
  	  \ds \int_{\Omega}U^2_{\epsilon}\log U^2_{\epsilon}&=\ds\int_{\Omega}\varphi^2 u^2_{\epsilon}\log \varphi^2+\ds\int_{\Omega}\varphi^2u^2_{\epsilon}\log u^2_{\epsilon}\\
  	   &\overset{\triangle}{=}\uppercase\expandafter{\romannumeral1}+\uppercase\expandafter{\romannumeral2} 
  	\end{array}$$
  	Since	
  		$\left|s^2\log s^2\right| \le C$  for $  0\le s \le 1,$ we have
  		$$\left|\uppercase\expandafter{\romannumeral1}\right| \le C\int_{\Omega}u^2_{\epsilon}=O(\epsilon^2).$$
  		$$
  II=\int_{B_{\rho}(0)}u^2_{\epsilon}\log u^2_{\epsilon}+\int_{\Omega\setminus{B_{\rho}(0)}}\varphi^2u^2_{\epsilon}\log u^2_{\epsilon}
  		 \overset{\triangle}{=}\uppercase\expandafter{\romannumeral2}_{1}+\uppercase\expandafter{\romannumeral2}_{2}.
  	$$
  	Since $\left|s\log s\right| \le C_1s^{1-\delta}+C_2 s^{1+\delta}$ for all $s>0$, where $0<C_1<C_2$ and $0<\delta<\frac{1}{3}$ such that $(N-2)(1-\delta)\geq 2$,
  	\begin{equation*}
    \begin{split}
  		 \left|\uppercase\expandafter{\romannumeral2}_{2}\right|&\le\int_{\Omega\setminus{B_{\rho}(0)}}\left|u^2_{\epsilon}\log u^2_{\epsilon}\right|\\
  		&\le C\int_{\Omega\setminus{B_{\rho}(0)}}(u^{2(1-\delta)}_{\epsilon}+ u^{2(1+\delta)}_{\epsilon})\\
  		&\le C\left|\Omega\right|(\epsilon^{(N-2)(1-\delta)}+\epsilon^{(N-2)(1+\delta)})\\
  		&=O(\epsilon^2),
	\end{split}
    \end{equation*}
and
  	\begin{equation*}
    \begin{split}
  		\uppercase\expandafter{\romannumeral2}_{1}
  		&=\int_{B_{(0,\rho)}}u^2_{\epsilon}\log u^2_{\epsilon}\mathrm{d}x\\  		 &=C\epsilon^2\int_{B_{\rho/\epsilon}(0)}\frac{1}{(1+\left|y\right|^2)^{N-2}}\log\left(C\epsilon^{-(N-2)}\frac{1}{(1+\left|y\right|^2)^{N-2}}\right)\mathrm{d}y\\
  		 &=C\epsilon^2\log(\frac{1}{\epsilon})\int_{B_{\rho/\epsilon}(0)}\frac{1}{(1+\left|y\right|^2)^{N-2}}
  +C\epsilon^2\int_{B_{\rho/\epsilon}(0)}\frac{1}{(1+\left|y\right|^2)^{N-2}}\log\frac{C}{(1+\left|y\right|^2)^{N-2}}\\
  		 &=C\epsilon^2\log(\frac{1}{\epsilon})\int_{\mathbb{R}^{N}}\frac{1}{(1+\left|y\right|^2)^{N-2}}\mathrm{d}y+O(\epsilon^2)
  +\epsilon^2O(\int_{\R^N}\frac{1}{(1+\left|y\right|^2)^{N-2-\frac{1}{4}}}\mathrm{d}y)\\  		 &=C\epsilon^2\log(\frac{1}{\epsilon})+O(\epsilon^2),
  	\end{split}
    \end{equation*}
  where we have used the fact that
    \begin{equation*}
        \int_{B^c_{\rho/\epsilon}(0)}\frac{1}{(1+\left|y\right|^2)^{N-2}}\mathrm{d}y=O(\epsilon^{N-4}).
    \end{equation*}
   Thus
    \begin{equation*}
        \int_{\Omega}U^2_{\epsilon}\log U^2_{\epsilon}=C_0\epsilon^2\log(\frac{1}{\epsilon})+O(\epsilon^2).
    \end{equation*}
    We complete the proof.
  \end{proof}
  \begin{lemma}\label{lm3.3}
  	If $N\ge 5$, $\lambda\in \R$  and $ \mu>0$, then $c_M< \frac{1}{N}S^{\frac{N}{2}}$.
  \end{lemma}
  \begin{proof}
  	Let $g(t)\overset{\triangle}{=}I(tU_\epsilon)$. By Lemma \ref{Lemma1}, $g(0)=0$ and  $\lim\limits_{t\to+\infty}g(t)$$=-\infty$, we can find $t_{\epsilon}\in(0,+\infty)$ such that
  	\begin{equation*}
 	 	 \sup\limits_{t\ge0}I(tU_\epsilon)=\sup\limits_{t\ge0}g(t)=g(t_{\epsilon})=I(t_\epsilon U_\epsilon).
  	\end{equation*}
	So
  	\begin{equation*}
  		\int\left|\nabla U_{\epsilon}\right|^2-t^{2^\ast-2}_{\epsilon}\int \left|U_{\epsilon}\right|^{2^\ast}-\lambda\int U_{\epsilon}^2-\mu\int U_{\epsilon}^2\log U^2_{\epsilon}-\mu\log t^2_{\epsilon}\int U^2_{\epsilon}=0,
  	\end{equation*}
  	which implies that, as $\epsilon\to0^+$
  	\begin{equation*}
      \begin{split}
  		2S^{\frac{N}{2}}
  		&\ge \int\left|\nabla U_{\epsilon}\right|^2-\lambda\int U_{\epsilon}^2-\mu\int U^2_{\epsilon}\log U^2_{\epsilon}\\
  		&= t^{2^\ast-2}_{\epsilon}\int \left|U_{\epsilon}\right|^{2^\ast}+\mu\log t^2_{\epsilon}\int U^2_{\epsilon}\\
  		&\ge t^{2^\ast-2}_{\epsilon}(\frac{1}{2}S^{\frac{N}{2}})-c\left|\log t^2_{\epsilon}\right|.
  	  \end{split}
    \end{equation*}
  	So there exists $c_{1}>0$ such that $t_{\epsilon}<c_{1}$.\\
  	On the  other hand, as $\epsilon\to 0^+$,
  	\begin{equation*}
      \begin{split}
  		\frac{1}{2}S^{\frac{N}{2}}
  		&\le\int\left|\nabla U_{\epsilon}\right|^2-\lambda\int U_{\epsilon}^2-\mu\int U^2_{\epsilon}\log U^2_{\epsilon}\\
  		&=t^{2^\ast-2}_{\epsilon}\int \left|U_{\epsilon}\right|^{2^\ast}_{\epsilon}+\mu\log t^2_{\epsilon}\int U^2_{\epsilon}\\
  		&\le 2S^{\frac{N}{2}}t^{2^\ast-2}_{\epsilon}+Ct^{2^\ast-2}_{\epsilon},
  	\end{split}
    \end{equation*}
  	which implies that there exists $c_{2}>0$ such that $t_{\epsilon}>c_{2}$.

  	Therefore, combining with the definition of $c_M$, we have that, as $\epsilon\to 0^+$,
  	\begin{equation*}
      \begin{split}
  		c_M&\le\sup\limits_{t\ge0}I(tu_\epsilon)\\
  		 &=\frac{t^2_{\epsilon}}{2}\int\left|\nabla U_{\epsilon}\right|^2-\frac{t^{2^\ast}_\epsilon}{2^\ast}\int \left|U_{\epsilon}\right|^{2^\ast}-\frac{\lambda}{2} t^2_{\epsilon}\int U^2_{\epsilon}-\frac{\mu}{2}\int t^2_{\epsilon}U^2_{\epsilon}(\log (t^2_{\epsilon}U^2_{\epsilon})-1)\\
  		 &\leq (\frac{t^2_{\epsilon}}{2}-\frac{t^{2^\ast}_\epsilon}{2^\ast})S^{\frac{N}{2}}+O(\epsilon^2)+\frac{\mu}{2}t^2_{\epsilon}(1-\log t^2_{\epsilon})\int U^2_{\epsilon}-\frac{\mu}{2}t^2_{\epsilon}\int U^2_{\epsilon}\log U^2_{\epsilon}\\
  		 &\le\frac{1}{N}S^{\frac{N}{2}}-c\mu\epsilon^2\log(\frac{1}{\epsilon})+O(\epsilon^2)\\
 		 &<\frac{1}{N}S^{\frac{N}{2}}.
  	 \end{split}
    \end{equation*}
  	We complete the proof.
  \end{proof}
      \textbf{The case for} $\mathbf{N=4:}$

      Let $\varphi(x)\in C^\infty_{0}(\Omega) $ be a radial function satisfying that   $\varphi(x)=1$  for $0\le\left|x\right|\le\rho$, $0\le\varphi(x)\le1$ for $\rho\le\left|x\right|\le2\rho$,  $\varphi(x)=0$  for $x\in \Omega\setminus B_{2\rho}(0)$,
  where $0<\rho\le1$  with  $\log(\frac{1}{8e^{3-\frac{\lambda}{\mu}}\rho^2})>1$.

  Set
  $$\begin{array}{ll}
  	U_{\epsilon}=\varphi(x)u_{\epsilon}(x).
 \end{array}$$

   	\begin{lemma}
   If $N=4$,  then, as $\epsilon \to 0^+$,
   	  $$\int_{\Omega}U^2_{\epsilon}\log U^2_{\epsilon}\geq 8\log\left(\frac{8(\epsilon^2+\rho^2)}{e(\epsilon^2+4\rho^2)^2}\right)\omega_4\epsilon^2\log(\frac{1}{\epsilon})+O(\epsilon^2)  $$
   and
   $$ \int_{\Omega}U^2_{\epsilon}\log U^2_{\epsilon}\leq 8\log\left(\frac{8 e(\epsilon^2+4\rho^2)}{(\epsilon^2+\rho^2)^2}\right)\omega_4\epsilon^2\log(\frac{1}{\epsilon})+O(\epsilon^2).$$
   \end{lemma}
  	  \begin{proof}Following from the definition of $U_\epsilon$, we have, as $\epsilon\to 0^+$,
   	  	\begin{equation}\label{3.4}
   \begin{split}
   	  		\int_{\Omega} U^2_{\epsilon}\log (U^2_{\epsilon})
   	  		 &=8\int_{\Omega}\varphi^2\left(\frac{\epsilon}{\epsilon^2+\left|x\right|^2}\right)^2\log\left[8\varphi^2(\frac{\epsilon}{\epsilon^2+\left|x\right|^2})^2\right]\mathrm{d}x\\
   	  		 &=8\int_{\Omega}\varphi^2\left(\frac{\epsilon}{\epsilon^2+\left|x\right|^2}\right)^2\log\left(\frac{\epsilon}{\epsilon^2+\left|x\right|^2}\right)^2\mathrm{d}x\\
   	  		 &+8\log(8)\int_{\Omega}\varphi^2\left(\frac{\epsilon}{\epsilon^2+\left|x\right|^2}\right)^2\mathrm{d}x\\
   	  		&+8\int_{\Omega\setminus {B_{\rho}(0)}}\varphi^2\left(\frac{\epsilon}{\epsilon^2+\left|x\right|^2}\right)^2\log\varphi^2\mathrm{d}x\\
   	  		 &=\uppercase\expandafter{\romannumeral1}_{1}+\uppercase\expandafter{\romannumeral1}_{2}+O(\epsilon^2).
   	  	  \end{split}
   \end{equation}
   By direct computation, we obtain
  	  	 \begin{equation}\label{bu3.5}
         \begin{split}
   	  	     \uppercase\expandafter{\romannumeral1}_{2}
   	  	     &=8\log8\int_{B_{\rho}(0)}\left(\frac{\epsilon}{\epsilon^2+\left|x\right|^2}\right)^2\mathrm{d}x+O(\epsilon^2)\\
   	  	     &=8\log8\omega_4\epsilon^2\int_{0}^{\rho/\epsilon}\frac{1}{(1+r^2)^2}r^3\mathrm{d}r+O(\epsilon^2)\\
   	  	     &=4\log8\omega_4\epsilon^2\left[\log(r^2+1)+\frac{1}{1+r^2}\right]{\Biggl\arrowvert}^{\rho/\epsilon}_{0}+O(\epsilon^2)\\
   	  	     &=4\log8\omega_4\epsilon^2\left[\log\left(\frac{\rho^2+\epsilon^2}{\epsilon^2}\right)+\frac{1}{1+\frac{\rho^2}{\epsilon^2}}-1\right]+O(\epsilon^2)\\
   	  	     &=8\log8\omega_4\epsilon^2\log(\frac{1}{\epsilon})+O(\epsilon^2),
   	  	  \end{split}
         \end{equation}
   	  	 \begin{equation}\label{bu3.6}
         \begin{split}
	  	    \uppercase\expandafter{\romannumeral1}_{1}	  	    &=8\int_{\Omega}\varphi^2\left(\frac{\epsilon}{\epsilon^2+\left|x\right|^2}\right)^2\log\left(\frac{\epsilon}{\epsilon^2+\left|x\right|^2}\right)^2\mathrm{d}r\\&=8\int_{B_{2\rho}(0)}\varphi^2\left(\frac{\epsilon}{\epsilon^2+\left|x\right|^2}\right)^2\log\left(\frac{\epsilon}{\epsilon^2+\left|x\right|^2}\right)^2\mathrm{d}x\\
   	  	    &=8\epsilon^2\int_{B_{{2\rho}/\epsilon}(0)}\varphi^2(\epsilon x)\frac{1}{(1+\left|x\right|^2)^2}\log\left(\frac{1}{\epsilon^2}\frac{1}{(1+\left|x\right|^2)^2}\right)\mathrm{d}x\\
   	  	    &=16\epsilon^2\log(\frac{1}{\epsilon})\int_{B_{{2\rho}/\epsilon}(0)}\varphi^2(\epsilon x)\frac{1}{(1+\left|x\right|^2)^2}\mathrm{d}x\\
  	  	    &\quad+8\epsilon^2\int_{B_{{2\rho}/\epsilon}(0)}\varphi^2(\epsilon x)\frac{1}{(1+\left|x\right|^2)^2}\log\frac{1}{(1+\left|x\right|^2)^2}\mathrm{d}x\\
   	  	    &\overset{\triangle}{=}\uppercase\expandafter{\romannumeral1}_{11}+\uppercase\expandafter{\romannumeral1}_{12},
   	  	  \end{split}
         \end{equation}
   where
  	  	  \begin{align}
 	  	  	\uppercase\expandafter{\romannumeral1}_{11}
   	  	  	&\ge 16\omega_4\epsilon^2\log(\frac{1}{\epsilon})\int_{0}^{\rho/\epsilon}\frac{1}{(1+r^2)^2}r^3\mathrm{d}r\notag\\
   	  	  	 &=8\omega_4\epsilon^2\log(\frac{1}{\epsilon})\left[\log(\frac{1}{\epsilon^2})+\log(\rho^2+\epsilon^2)+\frac{\epsilon^2}{\rho^2+\epsilon^2}-1\right]\notag\\
   	  	  	 &=16\omega_4\epsilon^2\left(\log(\frac{1}{\epsilon})\right)^2+8\omega_4\log\left(\frac{\rho^2+\epsilon^2}{e}\right)\epsilon^2\log(\frac{1}{\epsilon})
   	  	  	+O(\epsilon^4\log(\frac{1}{\epsilon})),\label{3.5}
   	  	  \end{align}
   	  	  \begin{align}
 	  	  	\uppercase\expandafter{\romannumeral1}_{11}
   	  	  	&\le 16\omega_4\epsilon^2\log(\frac{1}{\epsilon})\int_{0}^{2\rho/\epsilon}\frac{1}{(1+r^2)^2} r^3\mathrm{d}r\notag\\
   	  	  	 &=8\omega_4\epsilon^2\log(\frac{1}{\epsilon})\left[\log(\frac{1}{\epsilon^2})+\log(4\rho^2+\epsilon^2)+\frac{\epsilon^2}{4\rho^2+\epsilon^2}-1\right]\notag\\
   	  	  	&=16\omega_4\epsilon^2\left(\log(\frac{1}{\epsilon})\right)^2
   +8\omega_4\log\left(\frac{4\rho^2+\epsilon^2}{e}\right)\epsilon^2\log(\frac{1}{\epsilon})+O(\epsilon^4\log(\frac{1}{\epsilon})),\label{3.6}
   	  	  \end{align}
   \begin{align}
          	\uppercase\expandafter{\romannumeral1}_{12}
         	 &\ge-8\epsilon^2\int_{B_{{2\rho}/\epsilon}(0)}\frac{1}{(1+\left|x\right|^2)^2}\log(1+\left|x\right|^2)^2\mathrm{d}x\notag\\
          	 &=-16\omega_4\epsilon^2\int_{0}^{2\rho/\epsilon}\frac{1}{(1+r^2)^2}\log(1+r^2)r^3\mathrm{d}r\notag\\
          	 &=-8\omega_4\epsilon^2\int_{0}^{2\rho/\epsilon}\frac{r^2+1-1}{(1+r^2)^2}\log(1+r^2)\mathrm{d}(1+r^2)\notag\\
          	 &=-8\omega_4\epsilon^2\int_{0}^{2\rho/\epsilon}\frac{1}{1+r^2}\log(1+r^2)\mathrm{d}(1+r^2)\notag\\
          	 &+8\omega_4\epsilon^2\int_{0}^{2\rho/\epsilon}\frac{1}{(1+r^2)^2}\log(1+r^2)\mathrm{d}(1+r^2)\notag\\
         	&\ge - 4\omega_4\epsilon^2\left(\log (1+r^2)\right)^2{\Big\arrowvert}^{2\rho/\epsilon}_{0}\notag\\
          	&=-4\omega_4\epsilon^2\left(\log (1+\frac{4\rho^2}{\epsilon^2})\right)^2\notag\\
         	&=-4\omega_4\epsilon^2\left[\log (\epsilon^2+4\rho^2)+2\log(\frac{1}{\epsilon})\right]^2\notag\\
          	&=-16\omega_4\epsilon^2\left(\log(\frac{1}{\epsilon})\right)^2-16\omega_4\log (\epsilon^2+4\rho^2)\epsilon^2\log(\frac{1}{\epsilon})+O(\epsilon^2).\label{3.7}
          \end{align}
and
        \begin{align}
          	\uppercase\expandafter{\romannumeral1}_{12}
         	 &\le-8\epsilon^2\int_{B_{{\rho}/\epsilon}(0)}\frac{1}{(1+\left|x\right|^2)^2}\log(1+\left|x\right|^2)^2\mathrm{d}x\notag\\
          	 &=-16\omega_4\epsilon^2\int_{0}^{\rho/\epsilon}\frac{1}{(1+r^2)^2}\log(1+r^2)r^3\mathrm{d}r\notag\\
          	 &=-8\omega_4\epsilon^2\int_{0}^{\rho/\epsilon}\frac{r^2+1-1}{(1+r^2)^2}\log(1+r^2)\mathrm{d}(1+r^2)\notag\\
          	 &=-8\omega_4\epsilon^2\int_{0}^{\rho/\epsilon}\frac{1}{1+r^2}\log(1+r^2)\mathrm{d}(1+r^2)\notag\\
          	 &+8\omega_4\epsilon^2\int_{0}^{\rho/\epsilon}\frac{1}{(1+r^2)^2}\log(1+r^2)\mathrm{d}(1+r^2)\notag\\
         	&\le - 4\omega_4\epsilon^2\left(\log (1+r^2)\right)^2{\Big\arrowvert}^{\rho/\epsilon}_{0}+8\omega_4\epsilon^2\int_{0}^{\rho/\epsilon}\frac{1}{(1+r^2)}\mathrm{d}(1+r^2)\notag\\
         	&=-4\omega_4\epsilon^2\left[\log (\epsilon^2+\rho^2)+2\log(\frac{1}{\epsilon})\right]^2+8\omega_4\epsilon^2\left[\log (\epsilon^2+\rho^2)+2\log(\frac{1}{\epsilon})\right]\notag\\
          	&=-16\omega_4\epsilon^2\left(\log(\frac{1}{\epsilon})\right)^2-16\omega_4\log (\frac{\epsilon^2+\rho^2}{e})\epsilon^2\log(\frac{1}{\epsilon})+O(\epsilon^2).\label{3.8}
          \end{align}
        So, by \eqref{3.4}--\eqref{3.8}, we have that
          \begin{align*}
          	\int_{\Omega}U^2_{\epsilon}\log U^2_{\epsilon}
          &\ge 8\log8\omega_4\epsilon^2\ln(\frac{1}{\epsilon})+16\omega_4\epsilon^2\left(\log(\frac{1}{\epsilon})\right)^2+8\omega_4\log (\frac{\epsilon^2+\rho^2}{e})\epsilon^2\log(\frac{1}{\epsilon})\\
          	&-16\omega_4\epsilon^2\left(\log(\frac{1}{\epsilon})\right)^2-16\omega_4\log (\epsilon^2+4\rho^2)\epsilon^2\log(\frac{1}{\epsilon})+O(\epsilon^2)\\
          	 &=8\log\left(\frac{8(\epsilon^2+\rho^2)}{e(\epsilon^2+4\rho^2)^2}\right)\omega_4\epsilon^2\log(\frac{1}{\epsilon})+O(\epsilon^2)	
          \end{align*}
          and
         \begin{align*}
          	\int_{\Omega}U^2_{\epsilon}\log U^2_{\epsilon}
          &\le 8\log8\omega_4\epsilon^2\ln(\frac{1}{\epsilon})+16\omega_4\epsilon^2\left(\log(\frac{1}{\epsilon})\right)^2+8\omega_4\log (\frac{\epsilon^2+4\rho^2}{e})\epsilon^2\log(\frac{1}{\epsilon})\\
          	&-16\omega_4\epsilon^2\left(\log(\frac{1}{\epsilon})\right)^2-16\omega_4\log (\frac{\epsilon^2+\rho^2}{e})\epsilon^2\log(\frac{1}{\epsilon})+O(\epsilon^2)\\
          	&=8\log\left(\frac{8 e(\epsilon^2+4\rho^2)}{(\epsilon^2+\rho^2)^2}\right)\omega_4\epsilon^2\log(\frac{1}{\epsilon})+O(\epsilon^2).	
          \end{align*}
   	  \end{proof}
   \begin{lemma}\label{lm3.5}
   Assume that  $N=4$. If $ (\lambda, \mu)\in B_0\cup C_0$ and $\frac{32 e^\frac{\lambda}{\mu}}{\rho_{max}^2}< 1$, or $\lambda\in \R$ and $  \mu>0$ , then $c_M<\frac{1}{N}S^{\frac{N}{2}}$.
   \end{lemma}

      \begin{proof}
      	Let $g(t)\overset{\triangle}{=}I(tU_{\epsilon})$. Similar to the case of $N\geq 5$, we can find a  $t_{\epsilon}\in (0,+\infty)$ such that
     	\begin{align*}
            \sup_{t\ge0}I(tU_{\epsilon})=I(t_{\epsilon}U_{\epsilon})
      	\end{align*}
      	and
      	\begin{align*}
      	  \int\left|\nabla U_{\epsilon}\right|^2-t^{2^\ast-2}_{\epsilon}\int\left| U_{\epsilon}\right|^{2^\ast}-\lambda\int U_{\epsilon}^2-\mu\int U_{\epsilon}^2\log U^2_{\epsilon}-\mu\log t^2_{\epsilon}\int U^2_{\epsilon}=0.
      	\end{align*}
%
%
Similar to the case of $N\geq 5$ again,  we can see that there exists $0<C_2$ such that $t_\epsilon<C_2 $ for any $\mu\in \R\setminus\{0\}$ and there exists $C_1>0$ such that $t_\epsilon>C_1 $  for $\mu>0$.

\textbf{Case 1:} $\mu>0$

      So $$\mu\log t^2_{\epsilon}\int U^2_{\epsilon}=O(\epsilon^2\left|\ln\epsilon\right|),$$
and
      \begin{align*}
      	t^{2^\ast-2}_{\epsilon}
      	&=\frac{\int\left|\nabla U_{\epsilon}\right|^2-\lambda\int U_{\epsilon}^2-\mu\int U_{\epsilon}^2\log U^2_{\epsilon}-\mu\log t^2_{\epsilon}\int U^2_{\epsilon}}{\int\left|U_{\epsilon}\right|^{2^\ast}}\\
      	 &=\frac{S^{\frac{N}{2}}+O(\epsilon^2(\log(\frac{1}{\epsilon}))^2)}{S^{\frac{N}{2}}+O(\epsilon^N)}\longrightarrow 1 ~ as~ \epsilon\to 0^{+},
      \end{align*}
      which implies that,
      $$\mu\log t^2_{\epsilon}\int U^2_{\epsilon}=o(\epsilon^2\left|\ln\epsilon\right|),$$
      According to \eqref{bu3.5}, we get that
      $$\int_\Omega U_\epsilon^2=8\omega_4\epsilon^2\log(\frac{1}{\epsilon})+O(\epsilon^2).$$
      Therefore  we have that, as $\epsilon\to 0^+ ,$
      \begin{align*}
        c_M&\le I(t_{\epsilon}U_{\epsilon})\\
        &=\frac{t^2_\epsilon}{2}\int\left|\nabla U_{\epsilon}\right|^2
        -\frac{t^{2^\ast}_{\epsilon}}{2^\ast}\int\left| U_{\epsilon}\right|^{2^\ast}
        +\frac{\mu-\lambda}{2}t^2_{\epsilon}\int U^2_{\epsilon}
        -\frac{\mu}{2}t^2_{\epsilon}\int U^2_{\epsilon}\log U^2_{\epsilon}+o(\epsilon^2\left|\ln\epsilon\right|)\\
        &\le (\frac{t^2_\epsilon}{2}
        -\frac{t^{2^\ast}_{\epsilon}}{2^\ast})S^{\frac{N}{2}}
        +O(\epsilon^2) +\frac{\mu-\lambda}{2}t^2_{\epsilon}\int U^2_{\epsilon}
        -\frac{\mu}{2}t^2_{\epsilon}\int U^2_{\epsilon}\log U^2_{\epsilon}+o(\epsilon^2\left|\ln\epsilon\right|)\\
        &\le \frac{1}{N}S^{\frac{N}{2}}
        -\frac{t^2_{\epsilon}}{2}\int [\mu U^2_{\epsilon}\log U^2_{\epsilon}+(\lambda-\mu)U^2_{\epsilon}]
        +o(\epsilon^2\left|\ln\epsilon\right|)\\
        &\le \frac{1}{N}S^{\frac{N}{2}}
        -\frac{t^2_{\epsilon}}{2}\left(8\mu\log\left(\frac{8(\epsilon^2+\rho^2)}{e(\epsilon^2+4\rho^2)^2}\right)\omega_4\epsilon^2\log(\frac{1}{\epsilon})
        +(\lambda-\mu)8\omega_4\epsilon^2\log(\frac{1}{\epsilon})\right)
        +o(\epsilon^2\left|\ln\epsilon\right|)\\
        &\le \frac{1}{N}S^{\frac{N}{2}}
        -4 t^2_{\epsilon}\log\left(\frac{8^\mu(\epsilon^2+\rho^2)^\mu}{e^{2\mu-\lambda}(\epsilon^2+4\rho^2)^{2\mu}}\right)\omega_4\epsilon^2\log(\frac{1}{\epsilon})
        +o(\epsilon^2\left|\ln\epsilon\right|)\\
         &\le \frac{1}{N}S^{\frac{N}{2}}
        -4 t^2_{\epsilon}\log\left(\frac{8^\mu}{25^\mu e^{2\mu-\lambda}\rho^{2\mu}}\right)\omega_4\epsilon^2\log(\frac{1}{\epsilon})
        +o(\epsilon^2\left|\ln\epsilon\right|)\\
        &\le \frac{1}{N}S^{\frac{N}{2}}
        -4 C^2_{1}C\omega_4\epsilon^2\log(\frac{1}{\epsilon})+o(\epsilon^2\left|\ln\epsilon\right|)\\
        &<\frac{1}{N}S^{\frac{N}{2}},
      \end{align*}
where we choose $\rho>0$ small enough such that $\frac{8^\mu}{25^\mu e^{2\mu-\lambda}\rho^{2\mu}}>1.$

\textbf{Case 2: } $\mu<0$

When  $ (\lambda, \mu)\in B_0\cup C_0$, we choose $\rho=\rho_{max}$. We can see that $t_\epsilon \not \to 0$. Otherwise, $0<\alpha\leq c_M\le I(t_{\epsilon}U_{\epsilon})\to 0$, which is impossible.
Similar to Case 1, we can see that $t_\epsilon \to 1$ and $t_\epsilon^2\log t_\epsilon^2=o(1)$.
Therefore,
\begin{equation}\label{5221}
-\frac{\mu}{2}t_\epsilon^2\log t_\epsilon^2\int U_\epsilon^2=o(\epsilon^2\left|\ln\epsilon\right|).
\end{equation}

  Then we obtain that,  as $\epsilon\to 0^+ ,$
 \begin{align*}
        c_M&\le I(t_{\epsilon}U_{\epsilon})\\
        &=\frac{t^2_\epsilon}{2}\int\left|\nabla U_{\epsilon}\right|^2
        -\frac{t^{2^\ast}_{\epsilon}}{2^\ast}\int\left| U_{\epsilon}\right|^{2^\ast}
        +\frac{\mu}{2}t^2_{\epsilon}(1-\log t^2_{\epsilon})\int U^2_{\epsilon}
        -\frac{\mu}{2}t^2_{\epsilon}\int U^2_{\epsilon}(\log U^2_{\epsilon}+\frac{\lambda}{\mu})\\
        &\le (\frac{t^2_\epsilon}{2}
        -\frac{t^{2^\ast}_{\epsilon}}{2^\ast})S^{\frac{N}{2}}
        +\frac{\mu}{2} t^2_{\epsilon}\int U^2_{\epsilon}
        -\frac{\mu}{2}t^2_{\epsilon}\int U^2_{\epsilon}\log U^2_{\epsilon}-\frac{\lambda}{2}t^2_{\epsilon}\int U^2_{\epsilon}+o(\epsilon^2\left|\ln\epsilon\right|)\\
        &\le \frac{1}{N}S^{\frac{N}{2}}
        -\frac{\mu}{2}t^2_{\epsilon}\left( (\frac{\lambda}{\mu}-1)8\omega_4\epsilon^2\log(\frac{1}{\epsilon})+ 8\log\left(\frac{8 e(\epsilon^2+4\rho^2)}{(\epsilon^2+\rho^2)^2}\right)\omega_4\epsilon^2\log(\frac{1}{\epsilon})  \right)
        +o(\epsilon^2\left|\ln\epsilon\right|)\\
           &\le \frac{1}{N}S^{\frac{N}{2}}
        -4\mu \omega_4 t^2_{\epsilon}\left( (\frac{\lambda}{\mu}-1)+ \log\left(\frac{8 e(\epsilon^2+4\rho^2)}{(\epsilon^2+\rho^2)^2}\right)   \right)\epsilon^2\log(\frac{1}{\epsilon})
        +o(\epsilon^2\left|\ln\epsilon\right|)\\
         &\le \frac{1}{N}S^{\frac{N}{2}}
        -4\mu \omega_4 t^2_{\epsilon} \log\left(\frac{8 e^\frac{\lambda}{\mu}(\epsilon^2+4\rho^2)}{(\epsilon^2+\rho^2)^2}\right)   \epsilon^2\log(\frac{1}{\epsilon})
        +o(\epsilon^2\left|\ln\epsilon\right|)\\
        &\le \frac{1}{N}S^{\frac{N}{2}}
        -4\mu \omega_4  \log\left(\frac{32 e^\frac{\lambda}{\mu}}{\rho^2}\right)   \epsilon^2\log(\frac{1}{\epsilon})
        +o(\epsilon^2\left|\ln\epsilon\right|)\\
        &<\frac{1}{N}S^{\frac{N}{2}}
      \end{align*}
since the fact that $\frac{32 e^\frac{\lambda}{\mu}}{\rho_{max}^2}< 1$.

We complete the proof.

\end{proof}

\textbf{The case for} $\mathbf{N=3:}$

      Let $\varphi(x)\in C^1_{0}(\Omega) $ be a radial function satisfying  that  $\varphi(x)=1$  for $0\le\left|x\right|\le\rho$, $0\le\varphi(x)\le1$ for $\rho\le\left|x\right|\le2\rho$,  $\varphi(x)=0$  for $x\in \Omega\setminus B_{2\rho}(0)$,
  where $0<\rho$ is any fixed  constant such that $B_{2\rho}(0)\subset\Omega$ and $4\rho^2<1.$

  Set
  $$\begin{array}{ll}
  	U_{\epsilon}=\varphi(x)u_{\epsilon}(x).
  \end{array}$$
\begin{lemma}\label{lm3.6}
If $N=3$, then we have, as $\epsilon\to 0^+$,
\begin{equation}\label{3.8}
            \int_\Omega\left|\nabla U_\epsilon\right|^2\mathrm{d}x=S^\frac{3}{2}+\sqrt{3}\omega_3\int_{\rho}^{2\rho}\left|\varphi^{'}(r)\right|^2\mathrm{d}r\epsilon+O(\epsilon^3),
	    \end{equation}
	\begin{equation}\label{3.9}
		   \int_\Omega\left|U_\epsilon\right|^{2^*}\mathrm{d}x
	  =S^\frac{3}{2}+O(\epsilon^3),
   \end{equation}

   \begin{equation}\label{3.10}
	 	   \int_{\Omega}U^2_\epsilon \mathrm{d}x
	  =\sqrt{3}\omega_3\ds\int_{0}^{2\rho}\varphi^2~\mathrm{d}r\epsilon+O(\epsilon^2),
	 \end{equation}
and
\begin{equation}\label{3.11}
\int U^2_\epsilon\log U^2_\epsilon \mathrm{d}x =\sqrt{3}\omega_3\ds\int_{0}^{2\rho}\varphi^2~\mathrm{d}r\epsilon\log\epsilon+O(\epsilon),
\end{equation}
where $\omega_3$ denotes the area of the  unit sphere surface.
\end{lemma}
\begin{proof} Following from the definition of $U_\epsilon$, direct computations implies that
\begin{equation*}
   \begin{split}
   &\int_\Omega\left|\nabla U_\epsilon\right|^2\mathrm{d}x\\
	 &=\int_{B_{2\rho}}\left(\left|\nabla\varphi\right|^2\frac{\sqrt{3}\epsilon}{\epsilon^2+\left|x\right|^2}-2\nabla\varphi\cdot\varphi(r)\frac{\sqrt{3}\epsilon x}{(\epsilon^2+\left|x\right|^2)^2}+\frac{\sqrt{3}\epsilon\varphi^2(r) x^2}{(\epsilon^2+\left|x\right|^2)^3}\right)\mathrm{d}x\\
	 &=\sqrt{3}\omega_3\epsilon\int_{\rho}^{2\rho}\left|\varphi^{'}(r)\right|^2\frac{r^2}{\epsilon^2+r^2}\mathrm{d}r-2\sqrt{3}\omega_3\epsilon\int_{\rho}^{2\rho}\varphi^{'}(r)r\varphi(r)\frac{r^2}{(\epsilon^2+r^2)^2}\mathrm{d}r\\
	 &+\sqrt{3}\omega_3\epsilon\int_{0}^{\rho}\frac{r^4}{(\epsilon^2+r^2)^3}\mathrm{d}r+\sqrt{3}\omega_3\epsilon\int_{\rho}^{2\rho}\frac{\varphi^2(r) r^4}{(\epsilon^2+r^2)^3}\mathrm{d}r\\
	 &=\sqrt{3}\omega_3\epsilon\int_{\rho}^{2\rho}\left|\varphi^{'}(r)\right|^2\mathrm{d}r+O(\epsilon^3)
-2\sqrt{3}\omega_3\epsilon\int_{\rho}^{2\rho}\varphi^{'}(r)r\varphi(r)\frac{1}{\epsilon^2+r^2}\mathrm{d}r+O(\epsilon^3)\\
	&+\sqrt{3}\omega_3\epsilon\int_{0}^{+\infty }\frac{r^4}{(\epsilon^2+r^2)^3}\mathrm{d}r-\sqrt{3}\omega_3\epsilon\int_{\rho}^{+\infty }\frac{r^4}{(\epsilon^2+r^2)^3}\mathrm{d}r+\sqrt{3}\omega_3\epsilon\int_{\rho}^{2\rho}\frac{\varphi^2(r) r^4}{(\epsilon^2+r^2)^3}\mathrm{d}r\\
	 &=\sqrt{3}\omega_3\epsilon\int_{\rho}^{2\rho}\left|\varphi^{'}(r)\right|^2\mathrm{d}r-2\sqrt{3}\omega_3\epsilon\int_{\rho}^{2\rho}\varphi^{'}(r)\varphi(r)\frac{1}{r}\mathrm{d}r+O(\epsilon^3)\\
	&+\int_{\R^N}\left|\nabla u_\epsilon\right|^2-\sqrt{3}\omega_3\epsilon\int_{\rho}^{+\infty}\frac{1}{\epsilon^2 +r^2}\mathrm{d}r
+\sqrt{3}\omega_3\epsilon\int_{\rho}^{2\rho}\varphi^2(r)\frac{1}{\epsilon^2 +r^2}\mathrm{d}r+O(\epsilon^3)\\
	&=S^\frac{3}{2}+O(\epsilon^3)\\
&+\sqrt{3}\omega_3\epsilon\left(\int_{\rho}^{2\rho}\left|\varphi^{'}(r)\right|^2\mathrm{d}r-\int_{\rho}^{2\rho}2\varphi^{'}(r)\varphi(r)\frac{1}{r}\mathrm{d}r-\int_{\rho}^{+\infty}\frac{1}{r^2}\mathrm{d}r+\int_{\rho}^{2\rho}\varphi^2(r)\frac{1}{r^2}\mathrm{d}r\right)\\	 &=S^\frac{3}{2}+\sqrt{3}\omega_3\epsilon\int_{\rho}^{2\rho}\left|\varphi^{'}(r)\right|^2\mathrm{d}r+O(\epsilon^3)
\end{split}
\end{equation*}
%
	$$
	 	   \int_\Omega\left|U_\epsilon\right|^{2^*}\mathrm{d}x
	   =\int_{B_{\rho}(0)}\left|u_\epsilon\right|^{2^*}\mathrm{d}x+O(\epsilon^3)
=\int_{\mathbb{R}^3}\left|u_\epsilon\right|^{2^*}\mathrm{d}x+O(\epsilon^3)=S^\frac{3}{2}+O(\epsilon^3),
   $$
and
$$\begin{array}{ll}
	  \ds \int_{\Omega}U^2_\epsilon \mathrm{d}x
	   &=\sqrt{3}\ds\int_{B_{2\rho}(0)}\varphi^2\frac{\epsilon}{\epsilon^2+\left|x\right|^2}\mathrm{d}x\\
&= \sqrt{3}\omega_3\epsilon\ds\int_{0}^{2\rho}\varphi^2\frac{1}{\epsilon^2+r^2}r^2\mathrm{d}r\\
&=\sqrt{3}\omega_3\epsilon\ds\int_{0}^{2\rho}\varphi^2~\mathrm{d}r-\sqrt{3}\omega_3\epsilon^3\int_{0}^{2\rho}\varphi^2\frac{1}{\epsilon^2+r^2}~\mathrm{d}r\\
&=\sqrt{3}\omega_3\epsilon\ds\int_{0}^{2\rho}\varphi^2~\mathrm{d}r+O(\epsilon^2).
\end{array}$$

%
On the other hand, we have
   \begin{equation}\label{3.12}
     \begin{split}
     \int_{\Omega} U_\epsilon^2\log{U_\epsilon^2}&=\sqrt{3}\int_{\Omega}\varphi^2\frac{\epsilon}{\epsilon^2+\left|x\right|^2}\log(\sqrt{3}\varphi^2\frac{\epsilon}{\epsilon^2+\left|x\right|^2})\mathrm{d}x\\
       &=\sqrt{3}\int_{B_{2\rho}(0)}\varphi^2\frac{\epsilon}{\epsilon^2+\left|x\right|^2}\log(\sqrt{3}\varphi^2\frac{\epsilon}{\epsilon^2+\left|x\right|^2})\mathrm{d}x\\
       &=\sqrt{3}\omega_3\epsilon\int_{0}^{2\rho}\varphi^2(r)\frac{r^2}{\epsilon^2+r^2}\left[\log\sqrt{3}+\log\epsilon+\log\varphi^2+\log\frac{1}{\epsilon^2+r^2}\right]\mathrm{d}r\\
       &=\sqrt{3}\log\sqrt{3}\omega_3\epsilon\int_{0}^{2\rho}\varphi^2(r)\frac{r^2}{\epsilon^2+r^2}\mathrm{d}r\\
       &+\sqrt{3}\omega_3\epsilon\log\epsilon\int_{0}^{2\rho}\varphi^2\frac{r^2}{\epsilon^2+r^2}\mathrm{d}r\\
       &+\sqrt{3}\omega_3\epsilon\int_{0}^{2\rho}\varphi^2\log\varphi^2\frac{r^2}{\epsilon^2+r^2}\mathrm{d}r\\
       &+\sqrt{3}\omega_3\epsilon\int_{0}^{2\rho}\varphi^2(r)\frac{r^2}{\epsilon^2+r^2}\log\frac{1}{\epsilon^2+r^2}\mathrm{d}r\\
       &\overset{\triangle}{=}I_1+I_2+I_3+I_4.
     \end{split}
   \end{equation}

   By direct computation, we obtain that
   \begin{equation}\label{3.13}
   I_1=O(\epsilon),~I_3=O(\epsilon),
   \end{equation}
   \begin{equation}\label{3.14}
    \begin{array}{ll}       I_2
      & =\sqrt{3}\omega_3\epsilon\log\epsilon \ds\int_{0}^{2\rho}\varphi^2\mathrm{d}r-
       \sqrt{3}\omega_3\epsilon\log\epsilon \ds\int_{0}^{2\rho}\varphi^2\frac{\epsilon^2}{\epsilon^2+r^2}\mathrm{d}r\\
       &=\sqrt{3}\omega_3\epsilon\log\epsilon \ds\int_{0}^{2\rho}\varphi^2\mathrm{d}r+O(\epsilon^2\log\epsilon),\\
       \end{array}
        \end{equation}
and
   \begin{equation}\label{3.15}
     \begin{split}
       |I_4|
       &\leq -\sqrt{3}\omega_3\epsilon\int_{0}^{\rho} \log(\epsilon^2+r^2)\mathrm{d}r +O(\epsilon)\\
       &=-\sqrt{3}\omega_3\epsilon r\log(\epsilon^2+r^2){\Big\arrowvert}^{\rho}_{0}+\sqrt{3}\omega_3\epsilon\int_{0}^{\rho}r\frac{1}{\epsilon^2+r^2}2r\mathrm{d}r+O(\epsilon)\\
       &=O(\epsilon).
     \end{split}
   \end{equation}

%
   It follows from \eqref{3.12}--\eqref{3.15} that
   $$\int U^2_\epsilon\log U^2_\epsilon \mathrm{d}x =\sqrt{3}\omega_3\int_{0}^{2\rho}\varphi^2\mathrm{d}r\epsilon\log\epsilon+O(\epsilon). $$

We complete the proof.

   \end{proof}

   \begin{lemma}\label{lm3.71}
   	If $N=3 $ and $ (\lambda, \mu)\in B_0\cup C_0$, then  $c_M<\frac{1}{N}S^{\frac{N}{2}}$.
   \end{lemma}

\begin{proof}
Assume that $g(t):=I(tU_{\epsilon})$.
   Since $g(0)=0$, $\lim\limits_{t\to +\infty }g(t)=-\infty$ and  Lemma \ref{Lemma1}, we can get a $t_{\epsilon}\in (0,+\infty)$ such that
     	\begin{align*}
            \sup_{t\ge0}I(tU_{\epsilon})=I(t_{\epsilon}U_{\epsilon}).
      	\end{align*}

Similar to the case of $N=4,$  we can see that there exist $0<C_1<C_2$ such that $t_\epsilon\in (C_1, C_2).$
Therefore, for $\epsilon$ small enough,
      \begin{align*}
        c_M\le I(t_{\epsilon}U_{\epsilon})
        &=\frac{t^2_\epsilon}{2}\int\left|\nabla U_{\epsilon}\right|^2
        -\frac{t^{2^\ast}_{\epsilon}}{2^\ast}\int\left| U_{\epsilon}\right|^{2^\ast}
        -\frac{\mu}{2}t^2_{\epsilon}\int U^2_{\epsilon}\log U^2_{\epsilon}+O(\epsilon)\\
        &= (\frac{t^2_\epsilon}{2}
        -\frac{t^{2^\ast}_{\epsilon}}{2^\ast})S^{\frac{3}{2}}
        -\frac{\mu}{2}t^2_{\epsilon}\int U^2_{\epsilon}\log U^2_{\epsilon}+O(\epsilon)\\
        &\le \frac{1}{N}S^{\frac{3}{2}}
        -\frac{\sqrt{3}\mu}{2} C^2_{1}\omega_3\ds\int_{0}^{2\rho}\varphi^2~\mathrm{d}r\epsilon\log(\epsilon)+O(\epsilon)\\
        &<\frac{1}{N}S^{\frac{3}{2}}.
      \end{align*}
We complete the proof.
   \end{proof}

\section{The Proof of Main Theorems}\label{section:proofs}

  \begin{proof}[\textbf{The proof of Theorem \ref{th1}:}] Assume that $N\geq 4$ and  $(\lambda, \mu)\in A_0$.
   	   By Lemma \ref{Lemma1} and the mountain-pass theorem, there exists a sequence $\{u_n\}\subset H^1_{0}(\Omega)$ such that, as $n \to \infty$,
   	   \begin{align*}
   	   I(u_{n})&\rightarrow c_M,\\
   	   I^{'}(u_n)&\rightarrow 0,\quad in\quad (H^1_{0}(\Omega))^{-1},
   	   \end{align*}
   	    which, together with Lemma \ref{lm2.3}, implies that $\{u_n\}$ is bounded in $H^1_{0}(\Omega)$.
 	   By Lemmas \ref{lm2.4}, \ref{lm3.3} and \ref{lm3.5}, we can see that there exists $u\in H^1_{0}(\Omega)$ such that $u_n\to u$ in $H^1_{0}(\Omega)$,
   	   which implies that
   	   $$I(u)=c_M \quad and\quad I^{'}(u)=0.$$
   	Thus,
      $$   0=\langle I^{'}(u),u_{-}\rangle=\int\left|\nabla u_{-}\right|^2 , $$
       which implies that $u\ge 0$.\\
       \indent Therefore, $u$ is a nonnegative nontrivial  weak solution of $(1.1)$. By Moser's iteration, it is standard to prove that $u\in L^\infty(\Omega)$, then the H\"older estimate implies that $u\in C^{0,\gamma}(\Omega) ~(0<\gamma<1)$.
       Let $\beta:[0,+\infty)\mapsto \R$ be defined by $$\beta(s):=\begin{cases}\frac{3|\mu|}{2} |s\log s^2|,\quad& s>0,\\
       0, \quad&s=0,
       \end{cases}$$
       then for $a>0$ small enough, one can see that
       $$\Delta (u)=-u^{2^*-1}-\lambda u-\mu u \log u^2\leq \beta(u)~\hbox{in}~\{x\in\Omega: 0<u(x)<a\}.$$
       We may also assume that $a<\frac{1}{e}$, then $\beta'(s)=\frac{3|\mu|}{2}(-\log s^2 -2)>|\mu|(-\log a -1)>0$ for $s\in (0,a)$. So we have that $\beta(0)=0$ and $\beta(s)$ is nondecreasing in $(0,a)$. Furthermore,
       $$\int_{0}^{\frac{a}{2}}(\beta(s)s)^{-\frac{1}{2}}ds=-\sqrt{\frac{2}{3|\mu|}}(-2\log s)^{\frac{1}{2}}\Big|_{0}^{\frac{a}{2}}=+\infty.$$
       Hence, by \cite[Theorem 1]{v}, we have that $u(x) > 0$ in $\Omega$. In particular, for any compact $K\subset\subset\Omega$, there exists $c=c(K)>0$ such that $u(x)\geq c, \forall x\in K$. Take $K\subset\subset K_1\subset \subset \Omega $ and put  $f(x):=-u(x)^{2^*-1}-\lambda u(x)-\mu u(x)\log u(x)^2$, then $\Delta u=f(x)$ in $K_1$ and $f$ is of $C^{0,\gamma}$ in $K_1$. So by the standard Schauder estimate, we see that $u\in C^{2,\gamma}(K)$. By the arbitrariness  of $K$, we obtain that $u\in C^2(\Omega)$ and $u>0$ in $\Omega$.
      The proof is completed.
   \end{proof}

 \begin{proof}[\textbf{The proof of Theorem \ref{th2}:}]
   	   By Lemma \ref{Lemma1} and the mountain-pass theorem, there exists a sequence $\{u_n\}\subset H^1_{0}(\Omega)$ such that, as $n \to \infty$,
   	   \begin{align*}
   	   I(u_{n})&\rightarrow c_M,\\
   	   I^{'}(u_n)&\rightarrow 0,\quad in\quad (H^1_{0}(\Omega))^{-1},
   	   \end{align*}
   	    combining with Lemmas \ref{lm2.41}, \ref{lm3.5} and  \ref{lm3.71},  problem \eqref{1.1} has a nonnegative nontrivial  weak solution $u$.
   Applying a similar argument as the proof of Theorem \ref{th1}, we obatin that $u\in C^2(\Omega)$ and $u(x) > 0$ in $\Omega$.
   \end{proof}

\begin{proof}[\textbf{The proof of the Theorem \ref{buth1}:}]
Assume that problem \eqref{1.1} has a positive solution $u_0$ and let $\varphi_1(x)>0$ be the first eigenfunction corresponding to $\lambda_1(\Omega)$.  Then
$$\begin{array}{ll}
\ds\int_\Omega (u_0^{2^*-1}&+\lambda u_0+\mu u_0\log u^2_0)\varphi_1(x)
=\ds\int_\Omega -\Delta u_0\varphi_1(x)\\
&=\ds\int_\Omega -\Delta \varphi_1(x)u_0
=\ds\int_\Omega \lambda_1(\Omega)\varphi_1(x)u_0,\\
\end{array}$$
which implies that
\begin{equation}\label{bu4.81}
\ds\int_\Omega (u_0^{2^*-2}+\lambda-\lambda_1(\Omega)  +\mu\log u^2_0)u_0\varphi_1(x)=0.
\end{equation}
Define
$$f(s):=s^{2^*-2}+\mu\log s^2+\lambda-\lambda_1(\Omega), s>0,$$
then
$$f'(s)=(2^*-2)s^{2^*-3}+2\mu \frac{1}{s}.$$
By a direct computation, $f'(s)=0$ has a unique root $s_0=(-\frac{(N-2)\mu}{2})^\frac{N-2}{4}$. Furthermore,
$f'(s)<0$ in $(0,(-\frac{(N-2)\mu}{2})^\frac{N-2}{4})$ and $f'(s)>0$ in $((-\frac{(N-2)\mu}{2})^\frac{N-2}{4}, +\infty)$.
 Hence,
\begin{equation}\label{bu4.82}
f(s)\geq f((-\frac{(N-2)\mu}{2})^\frac{N-2}{4})=-\frac{(N-2)\mu}{2}+\frac{(N-2)\mu}{2}\log(-\frac{(N-2)\mu}{2})+\lambda-\lambda_1(\Omega)\geq 0.
\end{equation}

Since $u_0\in H^1_0(\Omega)$ and $u_0, \varphi_1>0$,  we have $\ds\int_\Omega f(u_0(x))u_0\varphi_1(x)>0$. Otherwise, $f(u_0(x))=0$ a.e in $\Omega$. That is,
$u_0(x)=(-\frac{(N-2)\mu}{2})^\frac{N-2}{4}$ a.e in $\Omega$, which contradicts to $u_0\in H^1_0(\Omega).$
By \eqref{bu4.81} and \eqref{bu4.82}, we have that
$$0=\ds\int_\Omega (u_0^{2^*-2}+\lambda-\lambda_1(\Omega)  +\mu\log u^2_0)u_0\varphi_1(x)=\ds\int_\Omega f(u_0(x))u_0\varphi_1(x)>0,$$
a contradiction. Hence,  problem \eqref{1.1} has no positive solutions.
\end{proof}

   \section*{Acknowledgement}
     This work is supported by the special foundation for Guangxi Ba Gui Scholars and the Natural
Science Foundation of China (Nos.11801581, 12061012, 11931012), Guangdong Basic and Applied Basic Research Foundation (2021A1515010034),Guangzhou Basic and Applied Basic Research Foundation(202102020225).



\begin{thebibliography}{99}

    \bibitem{bar}Barrios, B., Colorado, E., de Pablo, A., S\'anchez, U.. On some critical problems for the fractional Laplacian operator.  J. Differential Equations. 252 (2012), 6133-6162.
	
	\bibitem{bre1}Br\'ezis, H., Coron, J.. Multiple solutions of H-systems and Rellich's conjecture. Comm. Pure Appl. Math. 37 (1984),  149-187.
	
	\bibitem{bre2}Br\'ezis, H., Nirenberg, L.. Positive solutions of nonlinear elliptic equations involving critical Sobolev exponents. Comm. Pure Appl. Math. 36 (1983), 437-477.
	
	\bibitem{edm}Edmunds, D. E., Fortunato, D., Jannelli, E.. Critical exponents, critical dimensions and the biharmonic operator. Arch. Rational Mech. Anal. 112 (1990), 269-289.
			
	\bibitem{gao}Gao, F., Yang, M.. The Br\'ezis-Nirenberg type critical problem for the nonlinear Choquard equation. Sci. China Math. 61 (2018), 1219-1242.
	
	
	\bibitem{guy}Gu, Y., Deng, Y., Wang, X.. Existence of nontrivial solutions for critical semilinear biharmonic equations. Systems Sci. Math. Sci. 7 (1994), 140-152.
	
    \bibitem{lix}Li, X., Ma, S.. Choquard equations with critical nonlinearities. Commun. Contemp. Math. 22 (2020), 23-28.

	\bibitem{lie1}Lieb, Elliott H.. Sharp constants in the Hardy-Littlewood-Sobolev and related inequalities. Ann. of Math. 118 (1983), 349-374.
	
	\bibitem{lie2}Lieb, Elliott H., Loss, M.. Analysis. Graduate Studies in Mathematics, 14. American Mathematical Society, Providence, RI, 2001.

    \bibitem{lio}Lions, P. L.. Applications de la m\'ethode de concentration-compacit\`{e} $\grave{a}$ I  existence de fonctions extr\`{e}males. C. R. Acad. Sci. Paris S\`{e}r. I Math. 296 (1983), 645-648.
	
	\bibitem{poh}Poho\v aev, S. I.. On the Eigenfunctions of the equation $\Delta u+\lambda f(u)=0$. Dokl. Akad. Nauk SSSR.  165(1965), 36-39.
	
	\bibitem{shu}Shuai, W.. Multiple solutions for logarithmic Schr\"odinger equations. Nonlinearity 32 (2019), 2201-2225.
	
	\bibitem{tau1}Taubes, C. H.. The existence of a nonminimal solution to the SU(2) Yang-Mills-Higgs equations on $\mathbb{R}^3$. I. Comm. Math. Phys. 86 (1982), 257-298.

   \bibitem{tau2}Taubes, C. H.. The existence of a nonminimal solution to the SU(2) Yang-Mills-Higgs equations on $\mathbb{R}^3$. II. Comm. Math. Phys. 86 (1982), 299-320.
	 	
	\bibitem{uhl}Uhlenbeck, Karen K.. Variational problems for gauge fields. Ann. of Math. Stud. 102(1982), 455-464. 
		
	\bibitem{van}Van der Vorst, R. C. A. M.. Best constant for the embedding of the space $H^2\cap H^1_0$ into $L^{\frac{2N}{N-2}}$. Differential Integral Equations. 6 (1993), 259-276.

\bibitem{v}V$\acute{a}$zquez,  J..   A strong maximum principle for some quasilinear elliptic equations.  Appl. Math.
Optim. 12 (1984), 191-202.
	
	\bibitem{wil}Willem, M.. Minimax theorems. Progress in Nonlinear Differential Equations and their Applications, 24. Birkh\"auser Boston, Inc., Boston, MA, 1996.


    \end{thebibliography}
\end{document}